\documentclass{article}
\usepackage{CJKutf8}
\usepackage{definitions}

\newcommand{\lambdamax}{\lambda_{\mathrm{max}}}
\newcommand{\lambdamin}{\lambda_{\mathrm{min}}}
\newcommand{\zzref}{z_{\mathrm{ref}}}
\newcommand{\xref}{x_{\mathrm{ref}}}
\newcommand{\yref}{y_{\mathrm{ref}}}
\DeclareMathOperator*{\argmin}{arg\,min}
\newcommand{\prox}{\operatorname{prox}}

\usepackage{enumitem}
\usepackage{thmtools}
\usepackage{thm-restate}
\usepackage{hyperref}
\usepackage{tikz}
\usetikzlibrary{
    shapes.geometric,
    arrows.meta,
    positioning,
    fit
}

\tikzset{
  box/.style = {
    rectangle,
    rounded corners,
    minimum width=3.2cm,
    minimum height=1cm,
    text centered,
    draw=black,
    align=center
  },
  arrow_right/.style = {
    thick,
    -Implies,
    double equal sign distance
  },
  arrow_iff/.style = {
    thick,
    Implies-Implies,
    double equal sign distance
  }
}

\title{On the Local Linear Convergence of Operator Splitting Methods for Conic Programming}

\author{Lijun Ding\thanks{University of California San Diego, Department of Mathematics (\texttt{l2ding@ucsd.edu}).} \and Haihao Lu\thanks{MIT, Sloan School of Management (\texttt{haihao@mit.edu}).} \and Jinwen Yang\thanks{University of Chicago, Department of Statistics (\texttt{jinweny@uchicago.edu}).}}
\date{}

\begin{document}

\maketitle

\begin{abstract}
	Operator-splitting methods such as the primal-dual hybrid gradient method (PDHG) and the alternating direction method of multipliers (ADMM) often exhibit linear convergence on conic programs, although general theory guarantees only sublinear rates. We identify two geometric conditions -- strict complementarity and quadratic facial violation -- that explain this local behavior: under these conditions, PDHG and ADMM converge linearly to an optimal solution when initialized sufficiently close to the converging strictly complementary solution. We establish this result through a unified and verifiable primal-dual error-bound framework.
     First, we show that strict complementarity, together with a quadratic facial-violation property of the associated complementary faces, implies uniform quadratic growth of both the primal and dual augmented Lagrangians near a strictly complementary solution.   Second, we prove the local equivalence of three regularity conditions: uniform quadratic growth of the augmented Lagrangians, quadratic growth of a localized smoothed primal-dual gap, and metric subregularity of the saddle-point mapping. This equivalence clarifies the relationship among previously proposed conditions for local linear convergence. Third, using a unified formulation, we give a concise analysis showing that these equivalent conditions yield local linear convergence of PDHG and ADMM. We verify the quadratic facial-violation property for standard polyhedral and symmetric cones, as well as relevant faces of exponential and power cones, and show that it is preserved under Cartesian products. We also obtain an improved local rate using a restarted Halpern scheme.
     Finally, we extend the framework to convex composite optimization through a quadratic subdifferential-violation condition, which generalizes the quadratic facial-violation.
\end{abstract}


\medskip


\section{Introduction}
Conic programming offers a unified framework for modeling a broad class of convex optimization problems~\cite{Boyd_Vandenberghe_2004}, including 
 linear programming (LP), second-order cone programming (SOCP), and semidefinite programming (SDP) as special cases, and encompassing many formulations that arise across a wide range of applications \cite{bertsimas1997introduction,boyd1994linear,vandenberghe1996semidefinite,lobo1998applications,ding2021simplicity}. Concretely, given two {\color{black}finite-dimensional Euclidean spaces} $\mathbf{E}$ and $\mathbf{F}$ equipped with inner products, both denoted as $\inprd{\cdot}{\cdot}$, we consider a linear conic program in standard form with decision variable $x\in \mathbf{E}$:
\begin{equation}\label{eq: p} \tag{P}
    \begin{aligned}
        \min & \  \inprd{c}{x}
        \quad \mathrm{s.t.}\ \ Ax=b,\quad x\in \mathcal{K} \ , 
    \end{aligned}
\end{equation}
where the problem data consists of a linear map $A:\mathbf{E}\rightarrow \mathbf{F}$, a right-hand-side vector $b\in \mathbf{F}$, a cost vector $c\in\mathbf{E}$ and a convex and closed cone $\mathcal K$. Our result works for any cone, and we are particularly interested in cones and their dual cones in Table \ref{tab:cones}, or their finite Cartesian {\color{black}products}.
%
\begin{table}[ht!]
\centering
{
\begin{tabular}{l l l l l}
\toprule
\textbf{Cone} & \textbf{Notation} & \textbf{Dual Cone} 
& \textbf{Space} $\mathbf{E}$ & \textbf{Inner Product} $\inprd{\cdot}{\cdot}$ 
\\
\midrule
Zero cone & $\{0\}^d$ & $\mathbb{R}^d$ & $\mathbb{R}^d$ & dot product \\
Nonnegative cone & $\mathbb{R}_+^d$ & Self-dual 
& $\mathbb{R}^d$ & dot product \\
Second order cone & $\mathcal K_{\text{SOC}}^d$ & Self-dual & $\mathbb{R}^d$ & dot product\\
Positive semidefinite cone & $\mathbb{S}_+^d$ & Self-dual & $\mathbb{S}^d$ & trace product\\
Exponential cone & $\mathcal K_{\text{Exp}}^3$ & $\mathcal K_{\operatorname{D-Exp}}^3$ & 
$\mathbb{R}^3$ & dot product
\\
Power cone & $\mathcal K_{\text{Pow}}^3$ & $\mathcal K_{\operatorname{D-Pow}}^3$  & 
$\mathbb{R}^3$ & dot product
\\
\bottomrule
\end{tabular}}
\caption{\small Common cones and their duals. The mathematical definitions of $\mathcal K_{\text{SOC}}^d$, $\mathcal K_{\text{Exp}}^3$, $\mathcal K_{\text{Pow}}^3$ and their duals are as follows: $\mathcal K_{\text{SOC}}^d := \{(t, x) \in \mathbb{R}^{d} : \|x\| \le t\}$, $\mathcal K_{\text{Exp}}^3 := \cl(\{(u,v,w) \in \mathbb{R}^3 : v \exp(\frac{u}{v}) \le w, v > 0\})$, $\mathcal K_{\operatorname{D-Exp}}^3:=
\cl(\{(u,v,w)\in\mathbb R^3:-u\exp\pran{\tfrac{v}{u}}\leq ew, u<0\})$, $\mathcal K_{\text{Pow}}^3 := \{(u,v,w) \in \mathbb{R}^3 : u^p v^{1-p} \ge |w|,\;u,v\geq 0\}$ and $\mathcal K_{\operatorname{D-Pow}}^3:=\{(u,v,w)\in\mathbb R^3:({\tfrac{u}{p}})^p({\tfrac{v}{1-p}})^{1-p}\geq |w|,\;u,v\geq 0\}$ with $p\in (0,1)$. Here, $\cl(\cdot)$ denotes the closure operator. We denote $\mathbb{S}^d\subset \mathbb{R}^{d\times d}$ the set of symmetric matrices in $\mathbb{R}^{d\times d}$. The dot product for two vectors $x,y \in \mathbb{R}^d$ is $x^\top y =\sum_{i}x_iy_i$. The trace product of two matrices $A,B\in \mathbb{R}^{d\times d}$ is $\mathrm{tr}(A^{\top} B) = \sum_{ij} A_{ij} B_{ij}$.}
\label{tab:cones}
\end{table}
\\
The dual problem of \eqref{eq: p} is:
\begin{equation}\label{eq: d} \tag{D}
    \begin{aligned}
        \max_{y\in\mathbf{F},s\in \mathcal K^*} & \ -\inprd{b}{y} \quad
        \mathrm{s.t.}\ \ c+A^{\top} y=s \ , 
    \end{aligned}
\end{equation}
where $A^{\top}$ is the adjoint of $A$, and $\mathcal{K}^*$ is the dual cone of $\mathcal{K}$ with $\mathcal{K}^* = \{y\mid \inprd{y}{x}\geq 0, \;\text{for all}\;x\in \mathcal{K}\}$. 

To solve~\eqref{eq: p}, a wide range of solvers has been developed. Among them, interior-point methods (IPMs) have long been the dominant paradigm, owing to their robustness and strong convergence guarantees; they are implemented in widely used solvers such as MOSEK~\cite{andersen2003implementing,dahl2022primal} and Clarabel~\cite{goulart2024clarabel,chen2024cuclarabel}. More recently, advances in first-order methods (FOMs) have produced scalable alternatives to IPM-based solvers. Prominent examples include SCS~\cite{o2016conic,o2021operator}, COSMO~\cite{garstka2021cosmo}, and PDCS~\cite{lin2025pdcs}. These solvers rely on operator-splitting techniques applied to primal-dual reformulations of~\eqref{eq: p}: SCS and COSMO are based on variants of the alternating direction method of multipliers (ADMM)~\cite{gabay1976dual,glowinski1975approximation}, whereas PDCS is based on the primal-dual hybrid gradient method (PDHG)~\cite{zhu2008efficient,chambolle2011first}. In particular, these methods can be viewed as solving the saddle-point formulation (here, the indicator function $\iota_\mathcal{K}(x) = 0$ if $x\in \mathcal{K}$ and is $+\infty$ otherwise)
\begin{equation}\label{eq: pd} \tag{P-D}
    \min_x\max_y\; L(x,y):= \inprd{c}{x} +\iota_{\mathcal K}(x)+ \inprd{Ax-b}{y}\ .
\end{equation}


These {\color{black}solvers based on} operator-splitting {\color{black}techniques} typically exhibit strong practical performance in finding {\color{black}medium-accuracy and sometimes high-accuracy solutions}, due to \emph{local linear convergence} behavior observed in practice: once the {\color{black}iterates} enter some neighborhood of the solution, the convergence rate becomes linear  \cite{kang2025local,jiang2026local}. However, under merely standard feasibility and {\color{black}Slater-type} assumptions, most operator-splitting methods, including ADMM and PDHG, can only achieve sublinear convergence rates, typically $O(1/k)$ in the iteration count $k$, for measures such as the primal-dual gap and primal-dual infeasibility~\cite{he20121,davis2017faster,chambolle2011first,chambolle2016ergodic,lu2023unified}.

 To show linear convergence for operator-splitting methods, various stronger conditions {\color{black}have} been proposed in the literature. One prominent condition in monotone operator theory is the metric subregularity of certain {\color{black}subdifferentials, an error-bound} condition for saddle point formulations like \eqref{eq: pd}. It is well-known that metric subregularity implies local linear convergence for many splitting schemes. Another condition proposed for local linear convergence in the literature is the so-called quadratic growth of the smoothed gap \cite{tran2018smooth,fercoq2022quadratic}, which has been argued as broadly applicable as the metric subregularity condition. However, in conic programming, except for linear programming, verifying these properties is highly nontrivial,  due to the geometry of the feasible region boundary, the curvature of the cone, and the coupling between primal feasibility, dual feasibility, and complementarity. 
 
 Specialized to semidefinite programming, a representative case of conic programming (covering LP and SOCP as special cases), a recent remarkable result in~\cite {kang2025local} shows that ADMM achieves local linear convergence for SDP when the algorithm converges to an optimal solution that satisfies \emph{strict complementarity} (Definition \ref{def: scomp}), a regularity condition.  A later follow-up work \cite{jiang2026local} demonstrates the same holds for PDHG. However, the {\color{black}analyses in both works require} a rather intricate and detailed investigation of the semidefinite cone. Thus, a broader theory for general conic programs and related splitting schemes remains largely missing. 


Facing the insufficiency of existing theory in explaining the practical performance of operator splitting methods for conic programs, this work aims to advance the theoretical understanding by addressing the following question:
%
\begin{center}
\textit{Can we build a coherent, modularized, and unified framework for proving the local linear convergence? 
In particular, can we explain the relationship between the {\color{black}previously} proposed {\color{black}regularity conditions}, growth, and error bounds?}
\end{center}

\paragraph{Our contribution} Our answer is affirmative. Specifically, our main contributions are as follows:
\begin{itemize}
    \item \textbf{Identification of two geometric conditions for local linear convergence.} We identify two geometric conditions {\color{black}that} enable local linear convergence of operator splitting methods for conic programs: strict complementarity (Definition \ref{def: scomp}) and quadratic facial violation (Definition \ref{def: qfv}). As we discuss in detail in {\color{black}Remarks} \ref{rem: scomp} and \ref{rem: qfv} and Proposition \ref{prop:qfv_standard_cones}, they hold in a wide range of settings. In particular, strict complementarity (SC) holds unconditionally for LP, common generic settings \cite{dur2017genericity}, and many more structural applications \cite{ding2021simplicity} while quadratic facial violation (QFV) holds unconditionally for symmetric cones (including the first four cones in Table \ref{tab:cones}), and other cones under SC or generic assumptions. Our concrete theorem as a result of SC and QFV is the following:
    \begin{thm*}[Informal] 
        Consider the conic program~\eqref{eq: p} and its dual~\eqref{eq: d}, and suppose they admit a strictly complementary primal-dual solution and QFV. If PDHG or ADMM is initialized sufficiently close to (or converges to) this solution, then its iterates converge linearly to an optimal solution.
    \end{thm*}   
    These results show that, despite their worst-case global sublinear guarantees, operator splitting methods can converge much faster locally under strict complementarity and QFV. This provides a theoretical explanation for their strong empirical performance in conic programming, and covers previous local linear convergence results on SDP~\cite{kang2025local,jiang2026local} as a special case. We also discuss acceleration via restart and Halpern schemes in the same setting. 

\item \textbf{A verifiable primal-dual error-bound framework for establishing local linear convergence.} 
Apart from the concrete theorem, a central contribution of this paper is to provide a concrete way to verify the error bound or growth conditions that drive local linear convergence of operator splitting methods. 
Our analysis bridges this gap by deriving the needed conditions from strict complementarity (SC) and the local geometry of the cone, in particular, the quadratic facial-violation property (QFV). Our framework has three components. \textit{(I) Implication of SC and QFV:} In Proposition {\color{black}\ref{prop:uqg}}, we show that SC and QFV {\color{black}imply} uniform quadratic growth of the primal and dual augmented Lagrangian functions. This converts a classical conic regularity condition into quantitative primal and dual growth estimates, using tools from variational analysis and conic optimization~\cite{chan2008constraint,cui2016asymptotic,cui2019r,ding2023strict}, including bounded linear regularity~\cite{bauschke1999strong,zhang2000global}. \textit{(II) Equivalence of three conditions:} In Section \ref{sec:equiv}, we introduce an extended smoothed primal-dual gap, building on~\cite{tran2018smooth,fercoq2022quadratic}, that separates the primal and dual augmented-Lagrangian contributions, and in Theorem \ref{thm:equivalence}, we prove that uniform quadratic growth of the augmented Lagrangian functions, quadratic growth of the extended smoothed gap, and metric subregularity~\cite{dontchev2009implicit} are locally equivalent. Note that this result is independent of strict complementarity and QFV, and clarifies the relationships of the two conditions proposed in the literature for local linear {\color{black}convergence}. \textit{(III) A unified form for analyzing local linear convergence:} By revisiting a unified operator form \cite{lu2023unified}, in Proposition {\color{black}\ref{thm:linear-last}}, we provide a short proof that metric subregularity yields local linear convergence of PDHG and ADMM. Together, these results provide a verifiable path from strict complementarity and cone geometry to local linear convergence, as well as a {\color{black}clarification of the} relationship between past proposed conditions; see Figure~\ref{fig:proof} for a summary of the framework.

    \begin{figure}[ht!]
        \centering
        \begin{tikzpicture}[node distance = 0.6cm]
    
        \node (start) [box] {Strict complementarity and quadratic facial/subdifferential violation\\ of \eqref{eq: p} and \eqref{eq: d}};
        \node (qgal) [box, below=1cm of start] {Uniform QG of AL functions\\ with \eqref{eq: p} and \eqref{eq: d}};
        \node (phl) [box, draw=none, left=-1.8cm of qgal]{};
        \node (phr) [box, draw=none, right=-1.8cm of qgal]{};
        \node (qgsg) [box, below=of phl] {QG of smoothed gap \\ with \eqref{eq: pd}};
        \node (ms) [box, below=of phr] {Metric subregularity \\ with \eqref{eq: pd}};
    
        \node (group) [draw, rounded corners, inner sep=10pt, fit=(qgal) (qgsg) (ms)] {};
    
        \node (linear) [box, below=0.7cm of group] {Local linear convergence};
    
       \draw [arrow_right] (start) -- 
       node[midway, right, align=left, font=\small] {Proposition \ref{prop:uqg}}
       (group);
       
       \draw [arrow_iff] (qgal) -- (qgsg);
       \draw [arrow_iff] (qgsg) -- (ms); 
       \draw [arrow_iff] (qgal) -- (ms);
       
       \node[fill=white, inner sep=2pt, align=center, font=\small, yshift=0.2cm]
       at (barycentric cs:qgal=1,qgsg=1,ms=1)
       {Theorem \ref{thm:equivalence}};
       
       \draw [arrow_right] (group) -- 
       node[midway, right, align=left, font=\small] {Proposition \ref{thm:linear-last}}
       (linear);
        \end{tikzpicture}
        \caption{A framework for proving linear convergence of operator splitting methods}
        \label{fig:proof}
    \end{figure}
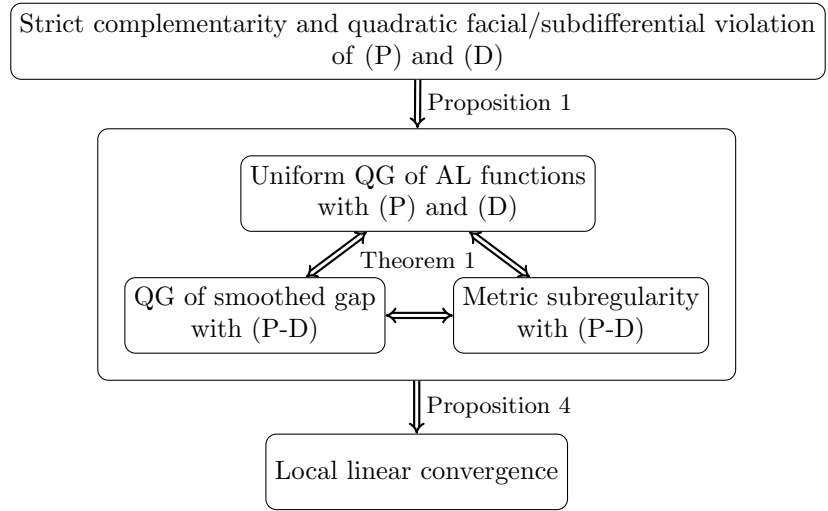

    \item \textbf{Beyond conic programming: convex composite programs.}
Our analysis is actually not tied to the conic programs alone. We further show that the same primal-dual mechanism extends to a class of convex composite programs by proper generalization of the two geometric conditions. The generalized strict complementarity and the generalized QFV, called quadratic subdifferential violation (QSV) again yield the local quadratic-growth and metric-subregularity properties needed for linear convergence. This extension indicates that the augmented-Lagrangian and smoothed-gap viewpoint provides a general template for proving local linear convergence of splitting methods beyond classical conic optimization. \textcolor{black}{We complement this view with an application to quadratic conic programs, covering a structural result from \cite{li2026gpu} on quadratic growth of the smoothed gap as a special case.}
\end{itemize}

\paragraph{Paper organization} The rest of the paper is organized as follows. In Section \ref{sec: prelim}, we introduce the basic {\color{black}notation and concepts} such as strict complementarity and QFV. In Section \ref{sec: uniformQG}, we derive that strict complementarity and QFV {\color{black}imply} uniform quadratic growth of both the primal and dual augmented Lagrangians near a strictly complementary solution. We also prove QFV holds for the faces of the {\color{black}cones of interest} in common settings. 
In Section \ref{sec:equiv}, we prove the local equivalence of the three regularity conditions: uniform quadratic growth of the augmented Lagrangians, quadratic growth of a localized smoothed primal-dual gap, and metric subregularity of the saddle-point mapping. In Section \ref{sec:linear}, we revisit a unified operator form of operator splitting method, and prove the local linear convergence of popular operator splitting methods under  the three conditions. {\color{black}We also present a} restarted Halpern variant {\color{black}that} improves the dependence on the local condition measure from quadratic to linear. In Section \ref{sec: generalization}, we generalize our framework to convex composite programming. We conclude the paper in Section \ref{sec: conclusion}.


\section{Preliminaries}\label{sec: prelim}
In this section, we introduce the structural conditions: strict complementarity and quadratic facial violation; the analytical object: augmented Lagrangian; the algorithms of interest: PDHG and ADMM; and the notations extensively used in the paper. We start with the concept of strict complementarity.

{\bf Basic notions and notations.}
Throughout the paper, denote $z=(x,y)$ the primal-dual pair, $\mathcal Z= \mathbf{E} \times \mathbf{F}$ the domain of $z$, and $\mathcal Z_\star= \mathcal {X}_\star \times \mathcal {Y}_\star$ the Cartesian product of the set of the primal optimal solutions and the set of dual optimal solutions. Given two sets $A,B\subset \mathbf{E}$, we denote $A+ B = \{v+y \mid v\in A, y\in B\}$. If $A = \{v\}$, we write $A + B = v+B$. We denote $\mathcal{S}_\star = \{c+A^{\top} y \mid y \in \mathcal{Y}_\star\}$ the set of dual slack optimal solutions.  For a set $C$, we denote $\iota_C$ to be indicator function of $C$, i.e., $\iota_C(x) = 0$ if $x\in C$ and $+\infty$ otherwise. We denote by $\|\cdot\|$ the $\ell_2$ norm induced by the inner product of the underlying space, and by $ B_r (x) $ the closed ball with center $x$ and radius $r$, defined as $B_r(x) = \{y \mid \|y-x\|\leq r \}$. For Cartesian product of linear spaces with inner products, the inner product of the Cartesian product is the sum of inner products. The operator norm of a linear operator, denoted as $\|\cdot\|$ as well, is induced by the $\ell_2$ norm. {\color{black}The} distance of a point $x$ to a set $C$ is $\dist_2(x,C) = \inf_{y\in C}\|x-y\|$. For a closed, proper, and convex function $f$ on a finite dimensional space $\mathbf{L}$, $f: \mathbf{L} \rightarrow\mathbb{R}\cup \{+\infty\}$ and $\lambda>0$, the convex conjugate is $f^*(y) = \sup_{x} \inprd{x}{y} - f(x)$, and 
the proximal mapping is $\prox_{\lambda h}(z)
    :=
    \argmin_u
    \left\{
        h(u)+\frac{1}{2\lambda}\|u-z\|^2
    \right\}.$ {\color{black}The domain of $f$} is defined as $\dom(f) = \{x\mid f(x) <+\infty\}${\color{black}.}
We denote the normal cone of $\mathcal{K}$ at an $x\in \mathcal{K}$ as $N_{\mathcal{K}}(x) = \{y\mid \inprd{y}{z-x}\leq 0,\text{ for all }z\in \mathcal{K} \}.$ In particular, $\partial \iota_{\mathcal{K}}(x) = N_{\mathcal{K}}(x)$.

\paragraph{Strict complementarity} 
To motivate strict complementarity, we recall the concepts of exposed faces in convex geometry. Given a convex and closed cone $\mathcal{K}$, an exposed face $\face_s$ defined by (or exposed by) an $s\in \mathcal{K}^*$ is the intersection of the cone and the hyperplane $s^\perp =\{x\mid \inprd{x}{s}=0\}$: 
\begin{equation}
    \face_s = \mathcal{K}\cap s^\perp.
\end{equation}
Similarly, we denote $\face_x^*$ to be an exposed face of the dual cone $\mathcal{K}^*$ defined by an $x\in \mathcal{K}$. To link this concept to the optimality conditions of \eqref{eq: p} and \eqref{eq: d}, recall that the complementarity condition of \eqref{eq: p} and \eqref{eq: d} for any optimal primal dual pair  $(\xsol,\ysol)$ with the dual optimal slack $\ssol = c+A^{\top} \ysol$ is
\begin{equation}\label{eq: cp}
\inprd{\xsol}{\ssol} = 0
\end{equation}
Hence, the complementarity condition combined with the conic condition that $\xsol \in \mathcal{K}, \ssol \in \mathcal{K}^*$ can be expressed as 
\begin{equation}
    \xsol \in \face_{\ssol} \quad \text{and}\quad 
    \ssol \in \face^*_{\xsol}.
\end{equation}
The strict complementarity condition requires a stronger inclusion: the $(\xsol,\ssol)$ lives in the relative interior of $\face_{\ssol}\times \face^*_{\xsol}$.
    
\begin{mydef}[Strict complementarity]\label{def: scomp}
The primal and dual programs, \eqref{eq: p} and \eqref{eq: d}, {\color{black}satisfy} strict complementarity if there is an optimal primal-dual pair $(\xsol,\ysol)$ of   \eqref{eq: p} and \eqref{eq: d} with slack $\ssol = c+A^{\top} \ysol$ such that 
\begin{subequations}\label{eq: sc_conic}
    \begin{align}
        \xsol \in \rel(\face_{\ssol}) \label{eq: dual_sc_conic} \\ 
        \ssol \in \rel(\face^*_{\xsol}) \label{eq: p_sc_conic} .
    \end{align}
\end{subequations}
Given two faces $\mathcal{C}_s\subset \mathcal{K}$ and $\mathcal{C}^{\textcolor{black}{*}}_x \subset \mathcal{K}^*$ satisfying $\inprd{x}{s}=0$, we say they {\color{black}satisfy} strict complementarity if \eqref{eq: sc_conic} is satisfied with $\xsol =x$ and $\ssol =s$.
\end{mydef}

\begin{rem}[Presence of strict complementarity]\label{rem: scomp}
Strict complementarity holds generically in the sense of \cite{dur2017genericity} regardless of the cone. It holds automatically for the cone of the nonnegative orthants \cite{goldman1956theory}. Though it may not always hold for other cones \cite{alizadeh1997complementarity}, it does hold in many structural instances \cite{ding2021simplicity,ding2024sharpnesswellconditioningnonsmoothconvex}. 

We note it is not always the case that \eqref{eq: dual_sc_conic} implies  \eqref{eq: p_sc_conic} and vice versa. However, it is the case if the cone and its dual cone are facially exposed  \cite[Remark 3.3.2]{pataki2000geometry} and \cite[Remark 4.10]{dur2017genericity}. Fortunately, all symmetric cones are facially exposed. Hence, the two conditions are equivalent for them.
\end{rem}

\paragraph{Quadratic facial violation} The following condition controls the distance of a point in a cone to an exposed face of the cone via violations of the defining equations, a form of error bound.
\begin{mydef}[Quadratic facial violation (QFV)]\label{def: qfv}
Given a vector $s\in \mathcal{K}^*$, the  exposed face $\face_s$ of a convex closed cone $\mathcal{K}$  defined by $s\in \mathcal{K}^*$ admits quadratic facial violation relative to $s$ if the following inequality holds: for any compact set $B$, there is a $\kappa > 0$ such that for any $x\in B\cap \mathcal{K}$,
\begin{equation}\label{eq: qfv_inequality}
    \dist^2(x,\face_s)\leq \kappa \inprd{s}{x}.
\end{equation}
We say an exposed face satisfies QFV if the face satisfies QFV relative to any defining vector. We also say complementary quadratic facial violation holds for faces $\face_s \subset\mathcal{K}$ and $\face_x^* \subset \mathcal{K}^*$ if the defining vectors satisfy $\inprd{x}{s}=0$ and both faces admit quadratic facial violation.
\end{mydef}
\begin{rem} [Presence of quadratic facial violation]\label{rem: qfv}
    As we verified in Section \ref{sec: Pqfv}, the quadratic facial violation condition is satisfied by all symmetric cones and the trivial cones unconditionally, the exponential cone under strict complementarity, and the power cone in a generic sense, thanks to the seminal work \cite{sturm2000error}, and a series of works on facial reduction and error bounds \cite{lourencco2021amenable,ding2023strict,lindstrom2023error,lin2024generalized}.
\end{rem}
\begin{rem}[Constant $\kappa$ dependence]\label{rem: qfv_issue}
    Note that the constant $\kappa$ in the above definition depends not only on the compact set $B$, but also on the particular defining vector $s$. In fact, if we change the defining vector of an exposed $\mathcal{C}_s$, then quadratic facial violation may not hold, e.g., \cite[Lemma 4.9]{lindstrom2023error} for the exponential cone. 
\end{rem}


\paragraph{Augmented Lagrangian} Recall the following primal and dual lagrangian of \eqref{eq: p} and \eqref{eq: d} with an optimal multiplier $(\xsol,\ysol)$:
\begin{equation}\label{eq: AL_conic}
    \begin{aligned}
        &L_{\beta}^P(x,u,y):=c^{\top}x+\iota_{\mathcal K}(x)+\iota_{\{b\}}(u)+\langle y, Ax-u\rangle+\frac{1}{2\beta}\|Ax-u\|^2\\
        &L_{\beta}^D(y,v,x):= \inprd{b}{y} +\iota_{\mathcal K^*}(c-v)+\langle x, -A^{\top} y-v\rangle+\frac{1}{2\beta}\|A^{\top} y+v\|^2 \ .
    \end{aligned}
\end{equation}
Here, the variable $v$ comes from adding the constraint $v = -A^{\top} y$ to \eqref{eq: d}.


\section{Uniform quadratic growth of augmented Lagrangian in conic programming and presence of QFV}\label{sec: uniformQG}

In this section, we establish a uniform quadratic growth property for the primal and dual augmented Lagrangians of conic programs (Proposition \ref{prop:uqg}). Under strict complementarity and the quadratic facial violation condition, we show that the augmented Lagrangian gap controls the squared distance to the corresponding primal and dual optimal solution sets, with constants that are uniform over all optimal primal-dual pairs in a sufficiently small neighborhood of a reference strictly complementary pair. This uniformity is the main technical point of the section and will be crucial for the local convergence analysis in the sequel. We then verify the quadratic facial violation condition for several cones of practical interest, including trivial cones, nonnegative cones, second-order cones, semidefinite cones, and selected exponential and power cone faces in Section \ref{sec: Pqfv}. Consequently, for these cones, strict complementarity alone implies the uniform quadratic growth of the augmented Lagrangian established in Proposition \ref{prop:uqg}.

 
\subsection{Uniform quadratic growth of augmented Lagrangian}
Proposition \ref{prop:uqg} is the main result of this section. It shows that strict complementarity, together with quadratic facial violation of the associated complementary faces, implies a uniform quadratic growth property for both the primal and dual augmented Lagrangians near a reference strictly complementary pair.
\begin{prop}[Uniform quadratic growth]\label{prop:uqg}
Suppose \eqref{eq: p} and \eqref{eq: d} {\color{black}admit} a {\color{black}strictly complementary} pair ${\zsol}_{0} = ({\xsol}_0,{\ysol}_0)$ and the corresponding complementary faces 
$\face^{\textcolor{black}{*}}_{{\xsol}_0}$ and $\face_{{\ssol}_0}$, where ${\ssol}_0 = c+A^{\top} {\ysol}_0$, admit quadratic facial violation. Fix $\beta>0$. Then there is an $r>0$, such that for any compact set $B\subset \mathbf{E}^2 \times \mathbf{F}^2$, there are constants $\kappa_1,\kappa_2$ so that the following inequality holds for any $(x,v,u,y)\in B$ and  any $\zsol=(\xsol,\ysol)\in \zsset \cap {\color{black}B_r({\zsol}_0)}$:
\begin{subequations}\label{eq: AL_diff_conic}
    \begin{align}
    L_{\beta}^P(x,u,\ysol) -  L_{\beta}^P(\xsol,b,\ysol) 
    &\geq \kappa_1 \dist^2((x,u), \mathcal{P}_\star) \label{eq: p_lag_Q_conic}\\ 
      L_{\beta}^D(y,v,\xsol) - L_{\beta}^D(\ysol, -A^{\top} \ysol,\xsol)
    & \geq \kappa _2 \dist^2((y,v),\mathcal{D}_\star), \label{eq: d_Lag_Q_conic} 
\end{align}
\end{subequations}
where $\mathcal{P}_\star = \{(x,Ax)\mid x\in \xsset\} $ and 
$\mathcal{D}_\star = \{(y, -A^{\top} y)\mid y \in \ysset\}.$
\end{prop}
\begin{rem}[Constant dependence and uniformity]
In the proposition statement, the constants $\kappa_1,\kappa_2$ do not depend on the specific choice of the optimizers $(\xsol,\ysol)$. They depend only on the set $B$, the radius $r$, the reference {\color{black}strictly complementary} pair ${\zsol}_{0} = ({\xsol}_0,{\ysol}_0)$, the penalty $\beta$, and the problem \eqref{eq: p} and \eqref{eq: d}. Hence, this is a uniform quadratic growth result, which is critical to our future equivalence conditions of metric subregularity with \eqref{eq: pd}.
\end{rem}

Our proof is based on three key ideas. First, the augmented Lagrangian optimality gaps (the left-hand side of \eqref{eq: AL_diff_conic}) can be interpreted as violations of the KKT conditions. Second, under strict complementarity, bounded linear regularity \cite{bauschke1996projection} becomes applicable, allowing us to bound the distance to the optimal solution set by two quantities: the linear constraint violation (one component of the KKT residual) and the distance to the complementarity face, which is closely related to the complementarity violation. Third, the complementarity violation in the KKT conditions controls the distance to the complementarity face through the quadratic facial violation condition. Applying the third idea, however, requires substantially more care than in existing analyses. Unlike prior work on quadratic growth for a single minimization problem \cite{drusvyatskiy2018error,cui2016asymptotic,ding2023strict}, where it suffices to analyze a fixed optimal solution, our setting requires uniform control over an entire set of optimal solutions, which is not handled by quadratic facial violation condition alone as discussed in Remark \ref{rem: qfv_issue}. Addressing this issue constitutes one of the main technical contributions of the paper. In particular, Lemma \ref{lem: inner_p_bound_diff_x_in_rel_face_conic} establishes a uniform bound on the constant in the quadratic facial violation condition across the relevant set of optimal solutions, making the above argument possible.


\begin{proof}[Proof of Proposition \ref{prop:uqg}]
We consider establishing the quadratic growth of the dual augmented Lagrangian, i.e., \eqref{eq: d_Lag_Q_conic}, only. The argument for \eqref{eq: p_lag_Q_conic} is almost identical, and we omit the details. We start with some preparations on the consequences of a small enough $r$.

\paragraph{Maintain complementary face and strict complementarity} We note first that since $({\xsol}_0,{\ssol}_0)$ (where ${\ssol}_0=c+A^{\top} {\ysol}_0$) lies in the relative interior of {\color{black}$\face_{{\ssol}_0} \times \face^{\textcolor{black}{*}}_{{\xsol}_0}$}, by choosing $r$ small enough, we can ensure that the face defined by $(\xsol,\ssol)$ (where $\ssol = c+ A^{\top} {\ysol}$) {\color{black}does} not change and {\color{black}the pair} also lies in the relative interior of the face. {\color{black}Specifically, choose $r_0>0$ small enough that every optimal pair $(\xsol,\ysol)\in B_{r_0}({\zsol}_0)$ satisfies $(\xsol,\ssol)\in\rel(\face_{{\ssol}_0})\times\rel(\face^{\textcolor{black}{*}}_{{\xsol}_0})$.} 
Then, for any $r<r_0$, we have for all $(\xsol,\ysol) \in  \zsset \cap {\color{black}B_r({\zsol}_0)}$
\begin{align}
\face^{\textcolor{black}{*}}_{{\xsol}_0}\times \face_{{\ssol}_0} 
 = 
\face^{\textcolor{black}{*}}_{{\xsol}}\times \face_{{\ssol}} \quad \text{and}\quad 
(\xsol,\ssol)
 \in \rel(\face_{{\ssol}})\times \rel(\face^{\textcolor{black}{*}}_{{\xsol}}).\label{eq: relative_interior}
\end{align}

\paragraph{Dual Lagrangian difference as KKT violations} 
We have 
\begin{align}
  {\color{black}L_{\beta}^D(y,v,\xsol) - L_{\beta}^D(\ysol, -A^\top \ysol,\xsol)} 
 & {\color{black}= \inprd{b}{y} +\iota_{\mathcal K^*}(c-v)+\langle \xsol, -A^{\top} y-v\rangle+\frac{1}{2\beta}\|A^{\top} y+v\|^2 -\inprd{b}{\ysol}} \\
 & {\color{black}\overset{(a)}{=} \inprd{A\xsol}{y}+\iota_{\mathcal K^*}(c-v)+\langle \xsol, -A^{\top} y-v\rangle+\frac{1}{2\beta}\|A^{\top} y+v\|^2 +c^\top \xsol} \\
 & {\color{black}= \inprd{\xsol}{c-v} + \iota_{\mathcal K^*}(c-v) + \frac{1}{2\beta}\|A^{\top} y+v\|^2 \label{eq: L_D_xs_linear_infeas_form_conic}}
\end{align}
Here, in the step $(a)$, we use the strong duality $-\inprd{b}{\ysol}  = 
c^\top \xsol$ (implied by strict complementarity) and $A\xsol = b$. If $c-v \not\in \mathcal{K}^*$, then \eqref{eq: d_Lag_Q_conic} trivially holds. Thus, we suppose that  $c-v \in \mathcal{K}^*$ below. Note now $L_{\beta}^D(y,v,\xsol) - L_{\beta}^D(\ysol, -A^\top \ysol,\xsol) $ is a sum of two nonnegative quantities and can be regarded as the violation of $(y,v)$ to the KKT conditions of \eqref{eq: d}.

\paragraph{Bounded linear regularity} Define the set $C = \{(y, s)\mid s = c+A^{\top} y\} \subset \mathbf{F} \times \mathbf{E}$ and the set $D = \mathbf{F} \times \face^{\textcolor{black}{*}}_{{\xsol}_0}$. 
Note that the intersection $C$ and $D$ 
is $C\cap D = \{(y,c+A^\top y)\mid y\in \ysset\}$, i.e., it is the set of optimal dual solutions of \eqref{eq: d}. Thanks to strict complementarity and \eqref{eq: relative_interior}, we know that 
\[
(\ysol, \ssol) \in C\cap \rel(D)\not=\emptyset 
\]
Using the above relationship and $C$ is an affine space, we can activate bounded linear regularity \cite[Corollary 5]{bauschke1999strong}, which states that 
for any $(y,s)$ in a bounded set $B_2$, there is {\color{black}a} constant $c_1>0$ such that
$\dist((y,s),C\cap D) \leq c_1( \dist((y,s),C) + \dist((y,s),D))$.
By recalling the definition of $C$ and $D$, we have that for any compact set $B_2 \subset \mathbf{F} \times \mathbf{E}$, there is an $\kappa'$ such that 
\begin{equation}\label{eq: blr_dual_side_conic}
\dist((y,s), \{(y,c+A^\top y)\mid y\in \ysset\}) \leq 
\kappa' (\norm{A^{\top} y +c -s} + \dist(s,\face^{\textcolor{black}{*}}_{\xsol})).
\end{equation}

\paragraph{Uniform quadratic facial violation} 
By comparing \eqref{eq: L_D_xs_linear_infeas_form_conic} and \eqref{eq: blr_dual_side_conic}, and with the {\color{black}substitution} $s = c-v$, we see that \eqref{eq: d_Lag_Q_conic} holds if the following uniform quadratic facial violation of $\face^{\textcolor{black}{*}}_{{\xsol}_0} = \face^{\textcolor{black}{*}}_{\xsol}$ holds: there exists a $r>0$ such that, for any compact $B_3\subset \mathbf{E}$, there is a $\kappa>0$ such that the following holds for any $s\in \mathcal{K}^* \cap B_3$ and $\xsol \in B_r({\xsol}_0)$: 
\begin{equation}\label{eq: qfv_dual_conic}
\dist^2(s,\face^{\textcolor{black}{*}}_{\xsol}) \leq \kappa \inprd{{\xsol}}{s}.
\end{equation}

Fortunately, this is accomplished by Lemma \ref{lem: inner_p_bound_diff_x_in_rel_face_conic} and our proof is complete. 
\end{proof}

\begin{lem}[Comparable inner product and uniform quadratic facial violation]
\label{lem: inner_p_bound_diff_x_in_rel_face_conic}
  Given a closed convex cone $\mathcal{K}$, a face $\face_0$ of $\mathcal{K}$,\footnote{{\color{black}Recall a face $\face$ of a convex cone $\mathcal{K}$ is a subset of $\mathcal{K}$ and satisfies that whenever $x,y\in \mathcal{K}$ and $0<\theta<1$ with $\theta x+(1-\theta)y\in\face$, we have $x,y\in\face$.}} and a point $x_0$ in the relative interior of $\face_0$. Then there is $r_1>0$ such that for any $s\in \mathcal{K}^*$ and $x_1\in B_{r_1}(x_0)\cap \face_0$, we have 
  \begin{equation}\label{eq: inner_p_bound_diff_x_in_rel_face}
      \inprd{s}{x_1}\geq \theta_1 \inprd{s}{x_0},
  \end{equation}
  where $\theta_1=1$ if $\face_0$ is a singleton or $x_0=0$, and $\theta_1 \geq \frac{r_1}{2\norm{x_0}}>0$ if $\face_0$ contains multiple points and $x_0\not=0$. Consequently, suppose a face $\mathcal{C}^{\textcolor{black}{*}}_{x_0}\subset \mathcal{K}^*$ is exposed by an $x_0$ and satisfies quadratic facial violation with respect to $x_0$. Then the uniform quadratic violation is satisfied, i.e.,
  there exists $r_2>0$ such that, for any compact set $B$, there exists $\kappa>0$ with the property that, for all 
  $x \in B_{r_2}(x_0) \cap \mathcal{C}_0$, the following inequality holds:
  \begin{equation} \label{eq: uniform_qfv}  \dist^2(s,\mathcal{C}^{\textcolor{black}{*}}_{x_0}) \leq 
  \kappa \inprd{s}{x},\quad \text{for all} \quad s\in \mathcal{K}^* \cap B.
  \end{equation}
\end{lem}
\begin{proof}
If $x_0=0$ or $\face_0$ is a singleton, the right-hand side is zero, and the assertion is immediate.

Suppose $x_0\ne0$.  Since $x_0\in\rel (\face_0)$, there exists $r>0$ such
that
\begin{equation}\label{eq: b_r_aff_C_ri_C}
B_{r}(x_0) \cap \aff(\face_0) \subset \rel(\face_0),
\end{equation}
where $\aff(\face_0)$ is the smallest affine space containing $\face_0$. Then, for the choice $r_1 = r/4$ and $\theta = \frac{r_1}{2\norm{x_0}}$, we have that for any {\color{black}$x_1\in B_{r_1}(x_0)\cap\face_0$}, the point 
$x_1-\theta  x_0\in \mathcal{K}$. Indeed, we have $x_1-\theta x_0\in \aff(\face_0)$ and $x_1 - \theta x_0 = x_0 + ((x_1-x_0)-\theta x_0) \in B_{r}(x_0)$. 
Thus $x_1-\theta  x_0\in \face_0\subset \mathcal{K}$ according to \eqref{eq: b_r_aff_C_ri_C}. 

Therefore, for every $s\in \mathcal{K}^*$, we have
\[
  0\le\langle s,x_1-\theta x_0\rangle
  =
  \langle s,x_1\rangle-\theta\langle s, x_0\rangle.
\]

The inequality \eqref{eq: uniform_qfv} then holds by taking $r_2 = r_1$ and $\kappa = \kappa_1/\theta>0$ where $\kappa_1$ is the quadratic facial violation parameter for $\mathcal{C}^{\textcolor{black}{*}}_{x_0}$ relative to $x_0$, i.e., $\dist^2(s,\mathcal{C}^{\textcolor{black}{*}}_{x_0})\leq \kappa_1 \inprd{s}{x_0}$ for all $s\in \mathcal{K}^*\cap B$.
\end{proof}
\subsection{Presence of quadratic facial violation in conic programming}\label{sec: Pqfv}

In this section, we verify this facial condition for several cone classes commonly used in conic optimization. The main message is that the quadratic facial violation property holds for all faces of the standard polyhedral and symmetric cones, and it also holds for the relevant exposed faces of exponential cones and for generic one-dimensional faces of power cones. Together with Proposition \ref{prop:uqg}, these results imply that, for conic programs over these cones, strict complementarity is sufficient to guarantee the uniform quadratic growth of the augmented Lagrangian whenever the corresponding complementary faces fall into the classes listed below.\footnote{Recall a cone $\mathcal{K}$ is proper if it is convex, closed, has nonempty interior, and is pointed \cite{Boyd_Vandenberghe_2004}.}

\begin{prop}[Quadratic facial violation for faces of standard cones]\label{prop:qfv_standard_cones}
Let $\face_s$ be an exposed face of a closed convex cone $\mathcal K$ with defining vector $s\in \mathcal K^*$. Then $\face_s$ admits quadratic facial violation in each of the following cases:
\begin{enumerate}
    \item $\mathcal K=\{0\}^{\dm}$ or $\mathcal K=\mathbb R^{\dm}$;
    \item $\mathcal K$ is a proper cone and $\face_s$ is either $\{0\}^{\dm}$ or $\mathcal K$;
    \item $\mathcal K=\mathbb R_+^{\dm}$, for all faces;
    \item $\mathcal K$ is a second-order cone, for all faces;
    \item $\mathcal K$ is a positive semidefinite cone, for all faces;
    \item $\mathcal K$ is a power cone or its dual, for the generic one-dimensional exposed faces described below.
\end{enumerate}
Also, $\face_s$ admits quadratic facial violation relative to the defining vector in the following case:
\begin{enumerate}[resume]
    \item $\mathcal K$ is an exponential cone or its dual, for the exposed faces with the defining vectors arising under strict complementarity, i.e., $\mathcal{C}_s$ admits a complementary face $\mathcal{C}^*_x$ and \eqref{eq: sc_conic} is satisfied for $\mathcal{C}_s$ and $\mathcal{C}^*_x$;
\end{enumerate}
Moreover, the quadratic facial violation property is preserved under finite Cartesian products of cones and faces in the above list.
\end{prop}

\begin{proof}
We prove the QFV property case by case. 
\paragraph{1. $\mathcal K=\{0\}^{\dm}$ or $\mathcal K=\mathbb R^{\dm}$.} If $\mathcal{K} = \{0\}^\dm$, then the only face is $\mathcal{K}$ itself. The inequality \eqref{eq: qfv_inequality} trivially holds as the only possible $x$ is $x=0$. If $\mathcal{K} = \mathbb{R}^\dm$, then the only face of it is itself. We also have $\mathcal{K}^* = \{0\}^\dm$, meaning the $s$ element in \eqref{eq: qfv_inequality} is $0$ and again 
\eqref{eq: qfv_inequality} trivially holds. 

\paragraph{2. $\mathcal K$ is a proper cone and $\face_s$ is either $\{0\}^{\dm}$ or $\mathcal K$.} Suppose that the cone $\mathcal{K}\subset\mathbb{R}^\dm$ is a proper cone. We {\color{black}verify \eqref{eq: qfv_inequality} for the two trivial faces}. First, for the trivial face $\face_s = \{0\}^\dm$. The defining vector is any element $s\in \intr(\mathcal{K}^*)$. Pick any $s\in \intr(\mathcal{K}^*)$, we have that $\inprd{s}{x}>0$ for any nonzero $x\in \mathcal{K}$. Hence, for the quantity $c=\inf_{\norm{x}=1,x\in \mathcal{K}} \inprd{s}{x}>0$ and for any $x\in \mathcal{K}$, we have 
\begin{equation}
    \dist(x,\face_s)= \norm{x} \leq \frac{1}{c} 
    \inprd{s}{x}.
\end{equation}

Hence, \eqref{eq: qfv_inequality} holds with $\kappa = \frac{1}{c^2}\sup_{x\in B\cap \mathcal{K}} \inprd{s}{x}$. Second, for the trivial face $\face_0 = \mathcal{K}$ itself, we must have $s=0$. In this case, both sides  \eqref{eq: qfv_inequality} are $0$ and \eqref{eq: qfv_inequality} holds. 

\paragraph{3. $\mathcal K=\mathbb R_+^{\dm}$ for all faces.} Given a defining vector $s\geq 0$, by letting $I_s = \{i\mid s_i>0\}$, the face $\face_{s} = \{x\mid x\geq 0 \;\text{and}\;x_i =0\;\text{for all}\;i\in I_s\}$. We assume, without loss of generality, that $ s\not=0$, since the $s=0$ case reduces to case 2 above. Then, for any $x\geq 0$, we have 
\begin{equation}
    \dist(x,\face_{s}) =\norm{x_{I_s}}\leq \sum_{i\in I_s} x_i \leq \frac{1}{\min_{i\in I_s} s_i}\inprd{s}{x}.
\end{equation}
Thus, the quadratic facial violation \eqref{eq: qfv_inequality} holds with $\kappa =  \frac{1}{\min_{i\in I_s} s_i^2}\sup_{x\in B\cap \mathcal{K}} \inprd{s}{x}$ for any compact $B$. Alternatively, quadratic facial violation \eqref{eq: qfv_inequality} can be established using Hoffman's lemma \cite{hoffman1952approximate}.

\paragraph{4. $\mathcal K$ is a second-order cone, for all faces.} The trivial faces have been considered before. For nontrivial face of $\mathcal{K}_{\text{SOC}}^d$, a defining vector $s = (r,y)$ satisfies $\norm{y} = r$ and the corresponding face is 
$\face_s = \{\lambda (r,-y)\mid \lambda \geq 0\}$. Define $\check{s} = \frac{1}{\norm{s}}(r,-y)$. Then, for any $x = (t,z) \in \mathcal{K}_{\text{SOC}}^d$, 
we find that 
\begin{equation}
\begin{aligned}
    \dist^2(x, \face_s) & = \norm{x - \inprd{\check{s}}{x}\check{s}}^2 = \norm{x}^2 - \inprd{\check{s}}{x}^2 = 
    t^2 +\|z\|^2 - \frac{\left( \inprd{s}{x}-2tr\right)^2}{\norm{s}^2} \\ 
    & \overset{(a)}{=}
    t^2 + \|z\|^2 - 2t^2 - \inprd{\frac{s}{\norm{s}}}{x}^2 + 2\sqrt{2}\frac{t}{\norm{s}}\inprd{s}{x} 
    \overset{(b)}{\leq} 
    2\sqrt{2} \frac{t}{\norm{s}} \inprd{s}{x},
\end{aligned}
\end{equation}
where the step $(a)$ is due to $\norm{y} = r=\frac{1}{\sqrt{2}}\norm{s}$ and the {\color{black}step} $(b)$ is because of $t\geq \norm{z}$. Hence, the quadratic facial violation \eqref{eq: qfv_inequality} holds with $\kappa =2\sqrt{2} \frac{\sup_{x\in B} \norm{x}}{\norm{s}}$ for any compact $B$ as $t\leq \norm{x}$ for $x\in \mathcal{K}^d_{\text{SOC}}$. 

\paragraph{5. $\mathcal K$ is a positive semidefinite cone, for all faces.} The trivial faces have been considered before. For nontrivial face of $\mathbb{S}_+^d$, a defining matrix $S\in \mathbb{S}_+^d$ has rank $r>0$ and the corresponding exposed face is 
$\face_S = \{X \mid X = VRV^\top, \;R\in \mathbb{S}_+^{d-r}\}$ where the matrix $V\in \mathbb{R}^{d \times {(d-r)}}$ has orthonormal columns and {\color{black}spans} the null space of $S${\color{black}.}
Note that the projection of any $X\in \mathbb{S}_+^d$ to the smallest linear space containing $\mathcal{C}_S$ is $X_V = VV^\top XVV^\top$ and $X_V \in \mathcal{C}_S$. Hence, we have 
\begin{equation}
    \dist(X,\face_S) = 
    \fronorm{X - X_V} \overset{(a)}{\leq} \frac{\inprd{S}{X}}{\lambda_{r}(S)} + \sqrt{\frac{2\inprd{S}{X}}{\lambda_{r}(S)}\lambda_1(X)},
\end{equation}
where $\lambda_r(S)$ is the $r$-th largest eigenvalue of $S$, and $\lambda_1(X)$ is the largest {\color{black}eigenvalue} of $X$, and the step $(a)$ is due to \cite[Lemma 4.3]{ding2021optimal}. Hence, by letting $B = \sup_{X\in B} \fronorm{X}$, the quadratic facial violation \eqref{eq: qfv_inequality} holds with {\color{black}$\kappa = 2\frac{\fronorm{S}B}{\lambda_r^2(S)} + 4\frac{B}{\lambda_{r}(S)}$} as $\lambda_1(X) \leq \fronorm{X}$ for $X\in \mathbb{S}_+^d$ and the Cauchy-Schwarz $\inprd{S}{X}\leq \fronorm{S} \fronorm{X}$.

{\color{black}The} above cones admit simple facial structures, and their quadratic facial violations can be established in a few lines. The cases of the exponential cone and the power cone (and their dual cones) are much more involved. 
Thankfully, a recent line of work \cite{lindstrom2023error,lin2024generalized,lin2025tight,wang2025error} has studied {\color{black}the} facial structure carefully.
They introduce a notion called one-step facial residual functions, which, in our context, when the point $x$ is restricted to a compact set and the cone, simplifies to an upper bound of $\dist(x,\mathcal{F})$ in terms of $\inprd{x}{s}$. We specify the situation of the quadratic facial violations for these cones, building on their work below.


\paragraph{6. $\mathcal K$ is an exponential cone or its dual, for the exposed faces arising under strict complementarity.}  Recall the exponential cone is $\mathcal{K} = \{(u,v,w) \in \mathbb{R}^3 : v \exp(\frac{u}{v}) \le w, v > 0\} \cup
\left\{
  (u,0,w):
  u\le 0,\; w\ge 0
\right\}$  and its dual cone $\mathcal{K}^*=\{(u,v,w)\in\mathbb R^3:-u\exp\pran{\tfrac{v}{u}}\leq ew, u<0\}\cup
\left\{
  (0,v,w):v\ge0,\ w\ge0
\right\}$ is simply a rotated and scaled version of it. More precisely, consider the invertible map $\Psi: (x,y,z)\in \mathbb{R}^3 \mapsto (-y,-x,\frac{z}{e})\in \mathbb{R}^3$, then $\mathcal{K}^* = \Psi(\mathcal{K})$. Note that quadratic facial violation is preserved under invertible linear transformations, up to changing constants on compact sets. Thus, we focus only on the exponential cone here. As described in \cite[Section 4.1]{lindstrom2023error}, there are three kinds of exposed faces. Thankfully, under strict complementarity, it can be quickly verified that all three {\color{black}types} of faces satisfy quadratic facial violation  under the corresponding strict complementary defining vector \cite[Corollary 4.4, Corollary 4.11 Item (i), and Corollary 4.7]{lindstrom2023error}; we omit the details.

\paragraph{7. $\mathcal K$ is a power cone or its dual, for the generic one-dimensional exposed faces.} Recall our power cone is $\mathcal K = \{(u,v,w) \in \mathbb{R}^3 : u^p v^{1-p} \ge |w|\}.$ with $p\in (0,1)$ and the dual cone $\mathcal K^*:=\{(u,v,w)\in\mathbb R^3:({\tfrac{u}{p}})^p({\tfrac{v}{1-p}})^{1-p}\geq |w|\}$, which is a scaled power cone under the map $\Psi: (u,v,w)\mapsto(up,v(1-p),w)$.
It is well-known that the power cone is a rotated second-order cone when $p=\frac{1}{2}$. Hence, let us focus on the power cone with $p\not=\frac{1}{2}$ and its proper nontrivial face. By \cite[Proposition 3.3 and Proposition 3.4]{lin2024generalized}, there are two types of nontrivial proper faces, both of which are exposed and are rays. The first kind contains only two rays: $E_1 = \mathbb{R}_+ (1,0,0)$ and $E_2 = \mathbb{R}_+(0,1,0)$, and they are actually complementary to each other. Moreover, according to \cite[Item (ii) of Corollary 3.9]{lin2024generalized}, they {\color{black}cannot} satisfy the quadratic facial violation simultaneously for $p\not=\frac{1}{2}$. The second kind, according to \cite[Proposition 3.3]{lin2024generalized}, has 
its defining vector $s = (s_x,1-p,s_z) \in \partial \mathcal{K}^*/\{0\}$ satisfying $(\frac{s_x}{p})^p  = |s_z|$ and the face is \begin{equation}
\face_s = \left\{t \left(\frac{p}{s_x},1,-\frac{1}{s_z}\right)\in \mathbb{R}^3 \mid t\geq 0\right \}.
\end{equation}
The complementary face in this case is 
$\face_x^* = \{t s\mid t\geq 0\}$ with any defining vector of the form $x = t \left(\frac{p}{s_x},1,-\frac{1}{s_z}\right)$ with $t>0$.
This type of face satisfies the quadratic facial violation using \cite[Item (i) Corollary 3.9]{lin2024generalized}. 

The faces $\mathcal{C}_s$ are generic in the following sense. 
If we consider the set of normalized proper faces $\mathcal{N} = \{ \frac{x}{\norm{x}}\mid \text{$x$ is in a proper face}\}$, then the one-dimensional Hausdorff measure of $\{\frac{x}{\norm{x}} \mid x\in E_1 \cup E_2\}$ is zero, as it is a set of two points, and the measure of $\mathcal{N}$ is the same as that of $\{\frac{x}{\norm{x}} \mid \text{$x$ in some $\mathcal{C}_s$}\}$. Note also that $\mathcal{C}_s$ converges to $E_1$ when $s_x \rightarrow 0^+$ and converges to $E_2$ when $s_x \rightarrow +\infty$. 



\paragraph{Cartesian product of cones.}
The quadratic facial violation property is preserved under finite Cartesian products of cones and faces in the above list. Suppose our cone is a finite Cartesian product of cones:
\begin{equation}
    \mathcal{K} = \times_{i=1}^q \mathcal{K}_i\quad {and}\quad \mathcal{K}_i \subset \mathbb{R}^{d_i}  \quad \text{for some positive integers $q$ and $d_i$.} 
\end{equation}
Then it is straightforward to verify that (i) the dual cone of $\mathcal{K}$ is $\mathcal{K}^* = \times_{i=1}^q \mathcal{K}_i^*,$ and (ii) if an exposed face $\face_{s}$ of $\mathcal{K}$ is defined by $s=(s_1,\dots,s_q)\in \mathcal{K}^*=\times_{i=1}^q \mathcal{K}_i^*$, then $\face_{s} = \times_{i=1}^q \face_{s_i}$ with $s_i$ as the defining vector of {\color{black}$\face_{s_i}$}. 
%
%
Thus, for any compact set $B \subset \times_{i=1}^q \mathbb{R}^{d_i}$, we have that for any {\color{black}$x=(x_1,\dots,x_q)\in \mathcal{K}\cap B$}, each component $x_i \in \mathcal{K}_i$ and is in a compact set in $\mathbb{R}^{d_i}$. Hence, for each $i$, there is a $\kappa_i$ such that 
\begin{equation}
    {\color{black}\dist^2(x_i,\face_{s_i}) \leq \kappa_i \inprd{s_i}{x_i}.}
\end{equation}
Finally, we have 
\begin{equation}
    {\color{black}\dist^2(x, \face_s) = \sum_{i=1}^q 
    \dist^2(x_i,\face_{s_i}) \leq \sum_{i=1}^q 
    \kappa_i \inprd{x_i}{s_i}\leq 
    \max_{1\leq i\leq q}\{\kappa_i\} \inprd{s}{x}.}
\end{equation}

Our proof is complete.
\end{proof}


\begin{rem}[QFV holds for faces of the PSD cone and symmetric cones]
The {\color{black}quadratic} facial violation for the PSD cone is implicitly proved in \cite{sturm2000error}. More generally, in \cite[Theorem 35]{lourencco2021amenable},  Lourenco has established that for any symmetric cone $\mathcal{K}$ and any exposed face $\face_s$ with a defining vector $s$, we have that for any compact set $B$, there is a number $\kappa'>0$ such that for any $x\in \mathcal{K}\cap B$, the following holds
 \begin{equation}
    \dist(x,\face_s) \leq \kappa'(\inprd{s}{x} + \sqrt{\inprd{s}{x}}).
 \end{equation} 
 It is then straightforward to verify that the quadratic facial violation condition holds for any \emph{symmetric cone}. We keep the more elementary (and hopefully simple enough) approach for particular cases of symmetric cones as it makes the paper more self-contained, and provides a more concrete description of the constant $\kappa$ in the quadratic facial violation condition.
\end{rem}

\paragraph{Additional cones} We refer readers to the line of work \cite{lindstrom2025optimal,lin2024generalized,lin2025tight} for the presence of quadratic facial violation on various faces and cones of practical interest. Though not all faces satisfy quadratic facial violation, we believe that when restricted to strict complementary faces or faces in common generic sense, quadratic facial violation does hold. We leave the {\color{black}detailed verification} to future work.

\section{Equivalence between uniform quadratic growth of AL functions and metric subregularity on \eqref{eq: pd}}\label{sec:equiv}

In this section, we establish an equivalence between three local regularity conditions for the primal-dual formulation \eqref{eq: pd}, spelled out in detail in Theorem \ref{thm:equivalence}: uniform quadratic growth of the primal and dual augmented Lagrangian functions, quadratic growth of a localized smoothed primal-dual gap, and metric subregularity of the saddle-point subdifferential. These three conditions arise naturally from different perspectives. The augmented Lagrangian growth condition is the form obtained from the conic-geometric analysis in Section \ref{sec: uniformQG}; the smoothed-gap growth condition is the condition used in primal-dual algorithmic analyses; and metric subregularity is the standard variational-analytic error bound for the saddle-point inclusion. The main result of this section shows that, locally around a reference optimal solution, these three viewpoints are equivalent, provided the quadratic growth constants are uniform over nearby optimal reference points.

As a consequence, the uniform augmented Lagrangian growth established in Proposition \ref{prop:uqg} can be transferred directly into metric subregularity of the saddle-point subdifferential, and hence into the error-bound condition used in the local linear convergence analysis in the next section.


In the following we start with the definition of metric subregularity and quadratic growth of localized smoothed gap, followed by main results of this section.
Throughout {\color{black}Sections} \ref{sec:equiv} and \ref{sec:linear}, we denote $f(x):=c^\top x+\iota_{\mathcal K}(x)$ and $g(u):=\iota_{\{b\}}(u)$, $f^*$ and $g^*$ as the convex conjugates of $f$ and $g$, respectively, and we define the 
\begin{equation}\label{eq: saddle_subdiff}
\mathcal F(x,y):=\begin{bmatrix}
        \partial f(x)+A^{\top} y \\ -Ax+\partial g^*(y)
    \end{bmatrix}=\begin{bmatrix}
        c+A^{\top} y+ N_{\mathcal K}(x) \\ b-Ax
    \end{bmatrix}
\end{equation}
as the sub-differential of primal-dual problem \eqref{eq: pd}. In particular, our augmented Lagrangian in \eqref{eq: AL_conic} becomes 
\begin{subequations}\label{eq: AL_gen}
\begin{align}
    L_{\beta}^P(x,u,y) & := f(x)+g(u) +\inprd{y}{Ax-u}+ \frac{1}{2\beta} \|Ax-u\|^2\quad \\
      L_{\beta}^D(y,v,x) &:=f^*(v)+g^*(y)+\inprd{x}{-A^{\top} y-v}+\frac{1}{2\beta}\|A^{\top} y+v\|^2 \ .
\end{align}
\end{subequations}


\begin{mydef}[Metric subregularity]\label{def:metric-subreg}
    Let ${\zsol}_0\in \zsset$ and $r>0$. We {\color{black}say that} \eqref{eq: pd} satisfies metric subregularity on {\color{black}the} region $B_r({\zsol}_0)\subset\mathcal Z$, if there exists a constant $\gamma > 0$ such that for all $z \in B_r({\zsol}_0)$,
    \begin{equation*}
        \gamma \operatorname{dist}_2(z, \zsset) \le \operatorname{dist}_2(0, \mathcal{F}(z)) \ .
    \end{equation*}
\end{mydef}

Next we introduce an intermediate notion, dubbed localized smoothed gap, that is crucial to bridge quadratic growth of AL functions with metric subregularity.
\begin{mydef}[Localized smoothed gap]
    Let $\beta>0$, $z_0^*\in \zsset$ and $r>0$. The localized smoothed gap of \eqref{eq: pd} is defined with $\zzref \in \mathcal{Z}\cap B_r({\zsol}_0)$
    and for any $z \in \mathcal Z\cap B_r({\zsol}_0)$ as 
    \begin{equation}\label{eq:eq-qg-gap}
        G_{\beta}(z; \zzref) = \max_{z'=(x',y') \in \dom(f)\times \dom(g^*)}\ L(x, y') - L(x', y) - \frac{\beta}{2} \|z' - \zzref\|^2 \ .
    \end{equation}
    Here, the maximum is understood as the supremum and is $+\infty$ if $x\not\in \dom(f)$ {\color{black}or} $y\not\in \dom(g^*)$.
\end{mydef}


The following theorem is the main result of this section. It shows that uniform quadratic growth of the reduced augmented Lagrangians, quadratic growth of the localized smoothed gap, and metric subregularity of the saddle-point mapping \(\mathcal F\) are locally equivalent. The uniformity is with respect to nearby optimal reference points, which is the form needed in the subsequent convergence analysis.
 
\begin{thm}\label{thm:equivalence}
    Consider the primal-dual problem \eqref{eq: pd}. Let ${\zsol}_0\in \zsset$. Then the following conditions are equivalent:
    \begin{enumerate}
        \item[(i)](Uniform quadratic growth of {reduced} AL functions) 
        There exists a radius $r>0$ such that the reduced primal and dual augmented Lagrangian functions satisfy uniform quadratic growth on $B_r({\zsol}_0)\subset\mathcal Z$. Specifically, there exist a constant $\beta>0$ and constants $\alpha_P:=\alpha_P(\beta)>0$ and $\alpha_D:=\alpha_D(\beta)>0$ such that, for all $\zsol=(\xsol,\ysol)\in\zsset\cap B_r({\zsol}_0)$ and for all $z=(x,y)\in B_r({\zsol}_0) \subset \mathcal{Z}$,
\begin{subequations}\label{eq:uniform-qg-reduced-al}
\begin{align}
L_\beta^P(x,\ysol)-L_\beta^P(\xsol,\ysol)
&\geq \frac{\alpha_P}{2}\dist_2^2(x,\xsset),\\
L_\beta^D(y,\xsol)-L_\beta^D(\ysol,\xsol)
&\geq \frac{\alpha_D}{2}\dist_2^2(y,\ysset),
\end{align}
\end{subequations}
where the reduced augmented Lagrangian functions are defined by
\begin{equation}\label{eq: minL_auxillary_uv}
L_\beta^P(x,\ysol):=\min_u L_\beta^P(x,u,\ysol),
\qquad
L_\beta^D(y,\xsol):=\min_v L_\beta^D(y,v,\xsol).
\end{equation}
        \item[(ii)](Quadratic growth of localized smoothed gap) There exists a radius $ r>0$ such that the smoothed gap of \eqref{eq: pd} satisfies quadratic growth on $B_{ r}({\zsol}_0)\subset\mathcal Z$, i.e., there exist a constant $\beta>0$ and a constant $\alpha:=\alpha(\beta) > 0$ such that for all ${\zsol} \in \zsset\cap B_{ r}({\zsol}_0)$ and for all $z \in {\color{black}B_r({\zsol}_0)} \subset \mathcal{Z}$,
        \begin{equation*}
            G_{\beta}(z; {\zsol}) \geq \frac{\alpha}{2} \operatorname{dist}_2(z, \zsset)^2\ .
        \end{equation*} 

        \item[(iii)](Metric {\color{black}subregularity}) There exists a radius $r>0$ such that \eqref{eq: pd} satisfies metric subregularity on $B_r({\zsol}_0)\subset\mathcal Z$, i.e., there exists a constant $\gamma>0$ such that for any $z\in  B_{r}({\zsol}_0)\subset \mathcal{Z}$, it holds that
        \begin{equation*}
            \gamma\dist_2(z,\zsset)\leq \dist_2(0,\mathcal F(z)) \ .
        \end{equation*}
    \end{enumerate} 
\end{thm}

\begin{rem}
    Quadratic growth with respect to the smoothed gap \eqref{eq:eq-qg-gap} was originally proposed in \cite{fercoq2022quadratic}, as an alternative condition ensuring local linear convergence of operator splitting type methods. Our result, Theorem \ref{thm:equivalence}, demonstrates that the quadratic growth of the smoothed gap \eqref{eq:eq-qg-gap} is equivalent to metric subregularity for primal-dual problems. 
\end{rem}

\begin{rem}
    We note that the constant $\alpha$ in the quadratic growth of the smoothed gap and the constants $\alpha_P$ and $\alpha_D$ of quadratic growth of reduced AL are uniform over all 
    ${\zsol} \in \zsset \cap B_r({{\zsol}}_{0})$ rather than for a single ${\zsol}$. 
\end{rem}

With Proposition \ref{prop:uqg} and Theorem \ref{thm:equivalence}, and noting that the minimizing auxiliary variables $u,v$ in the reduced augmented Lagrangians remain close to $z_{\star 0}$ when $z$ and $z_\star$ stay sufficiently close to $z_{\star 0}$ \footnote{Indeed, from \eqref{eq: AL_gen} and  \eqref{eq: minL_auxillary_uv}, we have 
$
u_\beta(x,y_\star)=\prox_{\beta g}(Ax+\beta y_\star),$ and $
v_\beta(y,x_\star)=\prox_{\beta f^*}(-A^\top y+\beta x_\star)$. So the claim follows from the nonexpansiveness of proximal mappings and the corresponding KKT identities at $z_\star$.},
we quickly conclude the following corollary. 
\begin{cor}[Strict complementarity and Quadratic facial violation imply metric subregularity]\label{cor: sc-ms}
{\color{black}Assume the hypotheses} of Proposition \ref{prop:uqg}, i.e.,
suppose \eqref{eq: p} and \eqref{eq: d} {\color{black}admit} a {\color{black}strictly complementary} pair ${\zsol}_{0} = ({\xsol}_0,{\ysol}_0)$ and the corresponding complementary faces 
$\face^{\textcolor{black}{*}}_{{\xsol}_0}$ and $\face_{{\ssol}_0}$, where ${\ssol}_0 = c+A^{\top} {\ysol}_0$, admit quadratic facial violation. {\color{black}Then} the three conditions described in Theorem \ref{thm:equivalence} hold. In particular,  \eqref{eq: pd} satisfies metric subregularity on $B_r({\zsol}_0)\subset\mathcal Z$ for some $r>0$.
\end{cor}

\subsection{Proof of Theorem \ref{thm:equivalence}}

In this section, we prove  Theorem~\ref{thm:equivalence}. We start with the following proposition.

\begin{prop}\label{prop: gap-computation}
Suppose $\beta>0$, ${\zsol}_0\in \zsset$ and $r>0$. Let $\zzref=(\xref,\yref)\in  B_r({\zsol}_0)\subset \mathcal{Z}$. Then it holds for any $z\in B_r({\zsol}_0)\subset \mathcal{Z}$ that
\begin{equation*}
    G_{\beta}(z;\zzref) = L_{\beta}^P(x,\yref) + L_{\beta}^D(y,\xref) \ .
\end{equation*}
where $L_{\beta}^P(x,\yref)$ and $L_{\beta}^D(y,\xref)$ are {reduced} augmented Lagrangian of primal problem \eqref{eq: p} and dual problem \eqref{eq: d} respectively and more specifically,
\begin{equation*}
    \begin{aligned}
        L_{\beta}^P(x,\yref) := \min_u\ L_{\beta}^P(x,u,\yref)
        \quad \text{and}\quad 
        L_{\beta}^D(y,\xref) := \min_v\ L_{\beta}^D(y,v,\xref)
    \end{aligned}
\end{equation*}
where $L_{\beta}^P(x,u,\yref)$ and $L_{\beta}^D(y,v,\xref)$ are defined in \eqref{eq: AL_gen}.
\end{prop}
\begin{proof}
Note that by definition, we have
    \begin{equation*}
    \begin{aligned}
        G_{\beta}(z; \zzref) &= \max_{z' \in \mathcal{Z}}\ L(x, y') - L(x', y) - \frac{\beta}{2} \|x' - \xref\|^2 - \frac{\beta}{2} \|y' - \yref\|^2\\
        &=\underbrace{\max_{y' \in \mathcal{Y}}\ L(x, y')- \frac{\beta}{2} \|y' - \yref\|^2}_{\mathrm{\bf (I)}} \ +\ \underbrace{\max_{x' \in \mathcal{X}}\ - L(x', y) - \frac{\beta}{2} \|x' - \xref\|^2}_{\mathrm{\bf (II)}} \ .
    \end{aligned} 
    \end{equation*}
    It holds for the first term that
    \begin{equation}\label{eq:smoothed-gap-I}
        \begin{aligned}
            {\mathrm{\bf (I)}} &=\ \max_{y' \in \mathcal{Y}}\ L(x, y')- \frac{\beta}{2} \|y' - \yref\|^2 = \max_{y' \in \mathcal{Y}}\ \left[f(x)+{y'}^\top Ax-g^*(y')\right]- \frac{\beta}{2} \|y' - \yref\|^2\\
            &= \ \max_{y' \in \mathcal{Y}}\ \left[f(x)+\inprd{y'}{Ax}-\max_u \pran{\inprd{u}{y'}-g(u)}\right]- \frac{\beta}{2} \|y' - \yref\|^2\\
            &\overset{(a)}{=} \ \min_u\ f(x)+g(u) +\sup_{y' \in \mathcal{Y}}\ \inprd{y'}{Ax-u}- \frac{\beta}{2} \|y' - \yref\|^2\\
            &= \ \min_u\ f(x)+g(u) +\inprd{\yref}{Ax-u}+ \frac{1}{2\beta} \|Ax-u\|^2 = \min_u\ L_\beta^P(x,u,\yref)
        \end{aligned}
    \end{equation}
    where the supremum is attained at $y'=\yref+\frac{1}{\beta}(Ax-u)$. Here, the step (a) is due to Fenchel duality, and the constraint qualification is satisfied as quadratic functions have the whole domain, and $g$ is proper, closed, and convex. 

    Similarly, we can rewrite the second term as
    \begin{equation}\label{eq:smoothed-gap-II}
        \begin{aligned}
            {\mathrm{\bf (II)}}&=\max_{x' \in \mathcal{X}}\ - L(x', y) - \frac{\beta}{2} \|x' - \xref\|^2 =\min_v\ f^*(v)+g^*(y)+\inprd{\xref}{-A^{\top} y-v}+\frac{1}{2\beta}\|A^{\top} y+v\|^2\\
            & =\min_v\ L_\beta^D(y,v,\xref) \ .
        \end{aligned}
    \end{equation}
   Combining \eqref{eq:smoothed-gap-I} and \eqref{eq:smoothed-gap-II}, we finish the proof.
\end{proof}

\begin{proof}[Proof of Theorem \ref{thm:equivalence}]
    \
    
    {\bf (a) Equivalence between uniform QG of AL functions and QG of localized smoothed gap.} This follows directly from definitions of AL functions and localized smoothed gap, combined with Proposition \ref{prop: gap-computation} with $\zzref = \zsol$.

{\color{black}{
\paragraph{(b) Quadratic growth of the localized smoothed gap implies
metric subregularity.}
Fix $z\in \mathcal Z\cap B_r({z_\star}_0)$ and let
$\widetilde z_\star
  :=
  \operatorname{proj}_{\mathcal Z_\star}(z).$ 
  Since ${z_\star}_0\in\mathcal Z_\star$, nonexpansiveness of the projection
gives $
  \|\widetilde z_\star-{z_\star}_0\|
  \le
  \|z-{z_\star}_0\|
  \le r.$ Hence the uniform smoothed-gap growth condition applies with center
$\widetilde z_\star$.

If $\mathcal F(z)=\emptyset$, then
$\operatorname{dist}_2(0,\mathcal F(z))=+\infty$ and the conclusion is
immediate. Otherwise, choose
\[
  w=(w_x,w_y)\in\mathcal F(z)
  \quad\text{such that}\quad
  \|w\|=\operatorname{dist}(0,\mathcal F(z)).
\]
For any $z'=(x',y')\in\mathcal Z$, the subgradient inequalities give $
  L(x,y')-L(x',y)
  \le
  \langle w,z-z'\rangle.$ 
Consequently,
\begin{align*}
  G_\beta(z;\widetilde z_\star)
  &\le
  \sup_{z'\in\mathcal Z}
  \left\{
    \langle w,z-z'\rangle
    -\frac{\beta}{2}\|z'-\widetilde z_\star\|^2
  \right\}
  &=
  \langle w,z-\widetilde z_\star\rangle
  +\frac{1}{2\beta}\|w\|^2
  &\le
  \|w\|\operatorname{dist}_2(z,\mathcal Z_\star)
  +\frac{1}{2\beta}\|w\|^2.
\end{align*}
Combining the above estimate with the quadratic growth condition yields
\begin{equation}\label{eq: qg_smooth_gap_msrg}
  \frac{\alpha}{2}\operatorname{dist}_2^2(z,\mathcal Z_\star)
  \le
  \|w\|\operatorname{dist}_2(z,\mathcal Z_\star)
  +\frac{1}{2\beta}\|w\|^2.
\end{equation}
If $\operatorname{dist}_2(z,\mathcal Z_\star)>0$ (otherwise, the local metric subregularity holds), divide \eqref{eq: qg_smooth_gap_msrg} by $\operatorname{dist}_2^2(z,\mathcal Z_\star)$ and
set  
\begin{equation} \label{eq: t_def}
t:=
  \frac{\|w\|}{\operatorname{dist}_2(z,\mathcal Z_\star)}.
  \end{equation}
Then \eqref{eq: qg_smooth_gap_msrg} becomes 
\[
  t^2+2\beta t-\alpha\beta\ge0,
\implies 
  t\ge \sqrt{\beta^2+\alpha\beta}-\beta.
\]
Thus, using \eqref{eq: t_def} and recalling $\|w\|=\operatorname{dist}(0,\mathcal F(z))$, we have
\[
  \left(\sqrt{\beta^2+\alpha\beta}-\beta\right)
  \operatorname{dist}_2(z,\mathcal Z_\star)
  \le
  \operatorname{dist}_2(0,\mathcal F(z)).
\]
Hence $\mathcal F$ is metrically subregular on
$B_r({z_\star}_0)$ with modulus
$\gamma
  =
  \sqrt{\beta^2+\alpha\beta}-\beta
  >0.$
}
}

{\bf (c) Metric subregularity implies QG of localized smoothed gap.} 
    Denote $\bar u=\prox_{\beta g}\pran{\beta \ysol+ Ax}$. and $y^+=\prox_{g^*/\beta }\pran{\ysol+ \frac{1}{\beta}Ax}\overset{(a)}{=}\ysol+ \frac{1}{\beta}Ax-\frac{1}{\beta}\bar u$, where the equality $(a)$ is due to Moreau identity. We have the following:
    \begin{equation}
        \begin{aligned}
            \partial_x L_{\beta}^P(x,\ysol)& = \partial_x \pran{\min_u \ f(x)+g(u) +\inprd{\ysol}{Ax-u}+ \frac{1}{2\beta} \|Ax-u\|^2}\\
            &\overset{(a)}{=}\partial f(x)+A^{\top} \ysol+\frac{1}{\beta}A^{\top} (Ax-\bar u)\\
            &= \partial f(x)+A^{\top} \pran{\pran{\ysol+ \frac{1}{\beta}Ax}-\frac{1}{\beta}\bar u}\\
            &= \partial f(x)+A^{\top} y^+.
        \end{aligned}
    \end{equation}
     In the step $(a)$, we use the fact that the Moreau envelope of $\beta g$ is smooth and convex. Hence, there is a unique gradient of the Moreau envelope of $\beta g$ at $\beta \ysol +Ax$ and the sum rule of subdifferential applies. 

    Recall ${y}^+=\prox_{g^*/\beta }\pran{\ysol+ \frac{1}{\beta}Ax}$. Note that
    \begin{equation*}
        \begin{aligned}
            &\ \ysol+\frac{1}{\beta}Ax\in y^++\frac{1}{\beta}\partial g^*(y^+)\quad
            \Rightarrow \quad \beta(\ysol-{y}^+)\in \partial g^*(y^+) - Ax \ ,
        \end{aligned}
    \end{equation*}
    Denote $z=(x,\ysol)$ and $z^+=(x,y^+)$. Requiring
    $z \in B_{\frac{r}{2+ \|A\|/\beta}}({\zsol}_0) $, we have $z^+ \in B_{r}({\zsol}_0)$. 
    \begin{equation*}
        \begin{aligned}
            \gamma \dist_2(x,\xsset)&=\gamma \dist_2(z,\zsset)\leq \gamma \|z-z^+\|+\gamma \dist_2(z^+,\zsset)=\gamma \|\ysol-y^+\|+\gamma \dist_2(z^+,\zsset)\\
            & \overset{(a)}{\leq} \gamma \|\ysol-y^+\|+\dist_2\pran{0,\begin{pmatrix}
                \partial f(x)+A^{\top} y^+ \\ \partial g^*(y^+)-Ax
            \end{pmatrix}}\\
            & \leq \gamma \|\ysol-y^+\|+\dist_2\pran{0,\begin{pmatrix}
                \partial f(x)+A^{\top} y^+ \\ \beta(\ysol - y^+)
            \end{pmatrix}}\\
            &\leq \gamma \|\ysol-y^+\|+ \pran{\dist_2\pran{0,
            \partial f(x)+A^{\top} y^+}+\beta\|\ysol-y^+\|}\\
            & \overset{\textcolor{black}{(b)}}{\leq} \dist_2\pran{0,
            \partial_x L_{\beta}^P(x,\ysol)
        }+\frac{2(\beta+\gamma)}{\sqrt{\beta}}\sqrt{L_{\beta}^P(x,\ysol)-L_{\beta}^P(\xsol,\ysol)}
        \end{aligned}
    \end{equation*}
Here, in the step $(a)$, we use the metric subregularity apply to $z^+$, and in the step $(b)$, we use the following fact:
    \begin{equation}
    \begin{aligned}
        L_\beta^P(x,\ysol) &= 
        \ L(x, y^+)- \frac{\beta}{2} \|y^+- \ysol\|^2 =
        \max_{y' \in \mathcal{Y}}\ L(x, y')- \frac{\beta}{2} \|y' - \ysol\|^2 \\
        & \overset{(i)}{\geq} L(x, \ysol)- \frac{\beta}{2} \|\ysol - \ysol\|^2+\frac{\beta}{2}\|y^+-\ysol\|^2 \overset{(ii)}{\ge} L(\xsol, \ysol)+\frac{\beta}{2}\|y^+-\ysol\|^2 \\
        & = L_\beta^P(\xsol, \ysol)+\frac{\beta}{2}\|y^+-\ysol\|^2
        \end{aligned}
    \end{equation}
    (i) is strong concavity and (ii) is due to saddle point property of $L$.

    Moreover, denote $v\in \partial_x L_{\beta}^P(x,\ysol)$ and $\txsol\in \xsset$ and by convexity,
\begin{equation*}
    \begin{aligned}
        {L_{\beta}^P(x,\ysol)-L_{\beta}^P(\xsol,\ysol)}\leq \langle v,x- \txsol\rangle\leq \|v\|\|x-\txsol\| \ .
    \end{aligned}
\end{equation*}
Thus due to the arbitrariness of $v$ and $\txsol$, we have
\begin{equation*}
    \begin{aligned}
        {L_{\beta}^P(x,\ysol)-L_{\beta}^P(\xsol,\ysol)}\leq \dist_2(0,\partial_x L_{\beta}^P(x,\ysol))\dist_2(x,\xsset) \ ,
    \end{aligned}
\end{equation*}
and therefore
\begin{equation}\label{eq:ms-al}
    \begin{aligned}
        \gamma\dist_2(x,\xsset)&\leq \dist(0,\partial_x L_{\beta}^P(x,\ysol))+\frac{2(\beta+\gamma)}{\sqrt{\beta}}\sqrt{\dist_2(0,\partial_x L_{\beta}^P(x,\ysol))\dist_2(x, \xsset)}. 
    \end{aligned}
\end{equation}

Denote $g:=\dist_2(0,\partial_x L_{\beta}^P(x,\ysol))$,  $\delta:=\dist_2(x,\xsset)$, 
$\mu:=\frac{2(\beta+\gamma)}{\sqrt{\beta}}$, and $\xi = \sqrt{g/\delta}$. 
We have from \eqref{eq:ms-al} that
\begin{equation*}
    \begin{aligned}
\gamma \delta \leq 
g + \mu \sqrt{g\delta}\Rightarrow \quad \xi^2+ \mu \xi-\gamma \geq 0\ , 
    \end{aligned}
\end{equation*}
Thus $\xi \geq \frac{-\mu +\sqrt{\mu^2+4\gamma}}{2}$ and it holds that
\begin{equation*}
     \dist_2(0,\partial_x L_{\beta}^P(x,\ysol))=g\geq \pran{\frac{-\mu +\sqrt{\mu^2+4\gamma}}{2}}^2 \delta=\pran{\frac{-\mu +\sqrt{\mu^2+4\gamma}}{2}}^2\dist_2(x,\xsset) \ ,
\end{equation*}
which means $L_\beta^P(x,\ysol)$ satisfies metric subregularity with parameter independent of $\ysol$, and thus quadratic growth of $L_\beta^P(x,\ysol)$~\cite{drusvyatskiy2013second,drusvyatskiy2015quadratic,drusvyatskiy2018error} with parameter independent  of $\ysol$. Similarly, we have quadratic growth of $L_\beta^D(y,\xsol)$. Combining quadratic growth of $L_\beta^P(x,\ysol)$ and $L_\beta^D(y,\xsol)$ with Proposition \ref{prop: gap-computation}, we achieve the quadratic growth of smoothed gap.
\end{proof}

\subsection{Necessity of conditions in Corollary \ref{cor: sc-ms}}

The following examples illustrate why the assumptions in Corollary \ref{cor: sc-ms} are needed for deriving metric subregularity from strict complementarity and facial growth conditions.

\subsubsection{Necessity of {\color{black}strictly complementary} pair for metric subregularity}
Consider SDP with $C=0$, $A_1 = \diag(1,0)$, $A_2 = \diag(0,1)$, and $b = (1,0)$. Note for SDP, the complementary quadratic facial violation is satisfied for free. 
The optimal primal solution is uniquely $\diag(1,0)$ and the set of optimal dual solutions contains $0$. One can verify that for $(X_\epsilon,0)$ with $X_\epsilon = \begin{bmatrix}
    1 & \epsilon \\ 
    \epsilon & \epsilon^2
\end{bmatrix}$, the metric subregularity of 
$\mathcal{F}$ is not satisfied for all small $\epsilon$. 
However, metric subregularity holds around $\xsol = \diag(1,0)$ and $\ysol = (0,1)$.

\subsubsection{Necessity of quadratic facial violation for metric subregularity}
Consider a power cone program with any $p>\frac{1}{2}$,  $c=\begin{bmatrix}
    0 \\ 1 \\ 0
\end{bmatrix}$, $A=\begin{bmatrix}
    0 & 0 & 0
\end{bmatrix}$, and $b= 0$. Using \cite[Proposition 3.3]{lin2024generalized}, the primal solution set is the following face: 
\begin{equation}
    \mathcal{X}_\star = \face_{s_\star} = \{(x,0,0)\mid x\geq 0\}
\end{equation}
where the defining vector as well as the dual optimal slack is $s_\star = (0,1,0)$, and the dual complementary face and the optimal dual solution set are 
\begin{equation}
\face_{\xsol}^* = \{(0,y,0)\mid y\geq 0\}\quad \text{and}\quad \mathbb{R}, \;\text{respectively},
\end{equation}
where $\xsol = (1,0,0)$. Note that strict complementarity is satisfied for the pair $(\xsol,\ysol)$ where $\ysol =0$ and $\ssol = c+A^{\top} \ysol=(0,1,0)$. 
We show that for this example, neither the quadratic facial violation nor the metric subregularity holds.

Consider the point $(x_\epsilon,0)$ where $x_\epsilon = (1,\epsilon, \epsilon^{1-p})$. The normal cone at $x_\epsilon$ is given by 
$N_\mathcal{K}(x_\epsilon) = 
\{-t(p, \frac{1-p}{\epsilon},-\epsilon^{p-1})\mid t\geq 0\}$. Let $s_0 = (p, \frac{1-p}{\epsilon},-\epsilon^{p-1})$. Since $\inprd{s_0}{s_\star} = (1-p)/\epsilon>0$, we find that 
\begin{equation}
    \dist_2^2(0, \mathcal{F}(x_\epsilon,0)) = 
    \norm{\ssol}^2 - \frac{\inprd{s_0}{\ssol}^2}{\norm{s_0}^2} = 
    \frac{p^2 \epsilon^2 + \epsilon^{2p}}{p^2 \epsilon^2 + (1-p)^2 + \epsilon^{2p}}.
\end{equation}
Thus, $\dist_2^2(0, \mathcal{F}(x_\epsilon,0))$ is on the order of $\epsilon^{2p}$ given $p>\frac{1}{2}$. On the other hand, we have 
\begin{equation}
    \dist_2^2 ((x_\epsilon,0),\mathcal{Z}_\star) = \dist_2 ^2 ( x_\epsilon, \face_{\ssol}) = \epsilon^2 + \epsilon^{2-2p}.
\end{equation}
Hence, $\dist_2^2 ((x_\epsilon,0),\mathcal{Z}_\star)$ is on the order $\epsilon^{2-2p}$. Combining the above expressions for $\dist_2^2 ((x_\epsilon,0),\mathcal{Z}_\star)$  and 
$\dist_2^2(0, \mathcal{F}(x_\epsilon,0))$, we see metric subregularity does not hold for $p>\frac{1}{2}$, no matter how small $\epsilon$ is. Moreover, the inner product $\inprd{\ssol}{x_\epsilon} = \epsilon$. Hence, the quadratic facial violation condition also fails.

\section{Local linear convergence on \eqref{eq: pd} and acceleration}\label{sec:linear}

This section connects the regularity theory developed above with algorithmic convergence. Using the metric subregularity obtained from strict complementarity and quadratic facial violation, we prove local linear convergence of PDHG and ADMM for solving \eqref{eq: pd}, described in detail below. 
Our proof is modular: using a unified operator-splitting form and metric subregularity, we derive Proposition \ref{thm:linear-last} in eight lines: the local linear rate for the last iterates. We then combine this proposition with Corollary \ref{cor: sc-ms} to derive the local linear rate under strict complementarity and QFV. Lastly, we further show how restart and Halpern acceleration improve the local rate dependence, based on our unified operator-splitting form.

{\bf PDHG and ADMM.} 
Given a convex closed set $C$, we denote the projection to $C$ as $\operatorname{proj}_C(x) = \arg\min_w \{\norm{x-w}\mid w\in C\}$. We describe two popular operator splitting methods for solving conic programming, thanks to their potentially low per-iteration cost:
\begin{itemize}
    \item PDHG: 
     Primal-dual hybrid gradient method (PDHG) has the update rule as follows when solving \eqref{eq: pd}:
    \begin{equation*}
        \begin{aligned}
            & x^+ \gets \operatorname{prox}_{\eta f}\pran{x-\eta A^{\top} y}\\
            & y^+ \gets \operatorname{prox}_{\eta g^*}\pran{y+\eta(2Ax^+-Ax)}
        \end{aligned}
    \end{equation*}
    where $\eta$ is the primal and dual stepsize typically chosen as $\eta<\frac{1}{\|A\|}$.
    \item ADMM: The alternating direction method of multipliers (ADMM) 
    when solving the dual problem \eqref{eq: d} is given as:
    \begin{equation*}
        \begin{aligned}
            x^{+}
            &\gets x-\rho\left(v^{+}+A^{\top}y\right),
            \;\text{where  $v^{+}
            \gets \operatorname*{arg\,min}_{v}
            \left\{
            f^{*}(v)-\langle x,v\rangle
            +\frac{\rho}{2}\left\|v+A^{\top}y\right\|_{2}^{2}
            \right\},$}
            \\[0.5em]
            y^{+}
            &\gets \operatorname*{arg\,min}_{u}
            \left\{
            g^{*}(u)-\langle Ax^{+},u\rangle
            +\frac{\rho}{2}\left\|v^{+}+A^{\top}u\right\|_{2}^{2}
            \right\}.
            \end{aligned}
    \end{equation*}
    where $\rho>0$ is the penalty parameter of ADMM. 
\end{itemize}

Note that we impose the following assumption on the surjectivity of the linear map $A$ to ensure convergence of the dual variable $y$:
    \begin{ass}\label{ass:full-rank}
    The linear map $A$ is surjective.
    \end{ass}
    
This assumption is commonly adopted in the ADMM literature, as it guarantees well-posedness of the dual update and facilitates convergence analysis. We make Assumption \ref{ass:full-rank} throughout this section, but this assumption is not required for the analysis of PDHG. 

\begin{thm}[Local linear convergence on solving \eqref{eq: pd}]\label{thm: local-linear}
{\color{black}Assume Assumption~}\ref{ass:full-rank} and the {\color{black}hypotheses} of Proposition \ref{prop:uqg}, i.e.,
suppose \eqref{eq: p} and \eqref{eq: d} {\color{black}admit} a {\color{black}strictly complementary} pair ${\zsol}_{0} = ({\xsol}_0,{\ysol}_0)$ and the corresponding complementary faces 
$\face^{\textcolor{black}{*}}_{{\xsol}_0}$ and $\face_{{\ssol}_0}$, where ${\ssol}_0 = c+A^{\top} {\ysol}_0$, admit quadratic facial violation. Consider $z^k=(x^k,y^k)$ from either ADMM or PDHG. Then there are $r,\tilde{r}>0$ and $q\in (0,1)$, such that for any $z^0= (x^0,y^0) \in B_r(z_{\star 0})$, the following inequality is satisfied 
  \begin{equation*}
    \mathrm{dist}_2(z^k,\zsset) \leq q^k \tilde{r}.
\end{equation*}
\end{thm}

\subsection{Proof of Theorem \ref{thm: local-linear}}\label{sec:linear-last}
In this section, we prove our main result, Theorem~\ref{thm: local-linear}. We first establish the local linear convergence of the PDHG and ADMM iterates under metric subregularity. Combining this result with Corollary~\ref{cor: sc-ms}, we then prove that PDHG and ADMM converge locally at a linear rate when applied to the conic program~\eqref{eq: p}.

Metric subregularity has been widely used to show linear convergence of various first-order methods for solving minimization problems~\cite{zhou2017unified,ye2021variational,drusvyatskiy2018error}, and, more recently, linear convergence of primal-dual algorithms~\cite{yuan2020discerning,fercoq2022quadratic,lu2022infimal,lu2025overview}. Here, we present a unified, simple approach to proving the linear convergence of PDHG and ADMM under metric subregularity, complementing existing, potentially sophisticated proofs. First, note that according to \cite{lu2023unified}, PDHG and ADMM are special instances of the generic algorithm with positive {\color{black}semidefinite} matrix $P$:
\begin{equation*}\label{eq: iteration}\tag{OSM}
    P(z^{k}-z^{k+1})\in \mathcal F(z^{k+1}) \ ,
\end{equation*}
where, noticing that $\eta$ is the stepsize of PDHG and $\rho$ is the penalty parameter of ADMM, we have
\begin{itemize}
    \item PDHG:
    $P=\begin{bmatrix}
        \frac{1}{\eta}I & -A^{\top} \\ -A & \frac{1}{\eta}I
    \end{bmatrix}$. This canonical matrix $P$ of PDHG is positive definite when stepsize $\eta<\frac{1}{\|A\|}$.
    \item ADMM:
    $P=\begin{bmatrix}
        \frac{1}{\rho}I & -A^{\top} \\ -A & \rho A A^{\top} 
    \end{bmatrix}$. This canonical matrix $P$ of ADMM is positive semi-definite for all penalty parameter $\rho>0$.
\end{itemize}


To facilitate our local convergence analysis, we introduce the following assumption:
{\color{black}\begin{ass}[Local retention]\label{ass:local}
There is a constant $C>0$ such that 
for any optimal solution ${\zsol}_0$ satisfying strict complementarity and any small enough $r>0$, the iterates $\{z^k\}_{k=0}^\infty$ of \eqref{eq: iteration} with initialization  $z^0\in B_r({\zsol}_0)$ satisfy
$z^k\in B_{\tilde r}({\zsol}_0)$ for all $k\geq 0$, where $\tilde r:=Cr$.
\end{ass}}
\begin{rem}
    Assumption~\ref{ass:local} states that when the algorithm is initialized near a strictly complementary solution, its iterates remain close to {\color{black}that solution}.
\end{rem}
The following proposition shows that both PDHG and ADMM indeed satisfy Assumption \ref{ass:local}, whose proof is deferred to Appendix \ref{app:prove-local}.
{\color{black}\begin{prop}\label{prop:local}
    Suppose that \eqref{eq: p} and \eqref{eq: d} admit a strictly complementary optimal pair. Consider PDHG with $\eta\|A\|<1$, or ADMM under Assumption~\ref{ass:full-rank}, for solving \eqref{eq: pd}.
    Then Assumption~\ref{ass:local} holds.
\end{prop}}





We now establish the linear convergence of the last iterates under the local {\color{black}retention} assumption (Assumption~\ref{ass:local}). Thanks to Proposition~\ref{prop:local}, this result applies directly to both PDHG and ADMM when solving problem~\eqref{eq: pd}.
\begin{prop}[Linear rate of last iterates]\label{thm:linear-last}
    Consider an iterate update rule \eqref{eq: iteration} with a positive {\color{black}semidefinite} matrix $P$ for solving the primal-dual form \eqref{eq: pd} of conic program \eqref{eq: p}. The iterates $\{z^k=(x^k,y^k)\}_{k=0,...,\infty}$ are obtained from this iterate update rule and the initial solution $z^0=(x^0,y^0)$ is chosen to satisfy Assumption \ref{ass:local}, where ${\zsol}_0$, $r$ and $\tilde r:=Cr$ be corresponding constants in Assumption \ref{ass:local}. Suppose metric subregularity is satisfied for \eqref{eq: pd} with constant $\gamma>0$ in the ball $B_{\tilde r}({\zsol}_0) \subseteq \mathcal{Z}$. Then it holds for any $k\ge 0$ that 
    \begin{equation*}
    \mathrm{dist}_2(z^k,\zsset) \leq \exp\pran{1-\frac{k}{\left\lceil e^2{\lambdamax^2(P)}/{\gamma^2}\right\rceil}}\tilde r \ .
\end{equation*}
\end{prop}
\begin{proof}
{\color{black}Let $N:=\left\lceil e^2\lambdamax^2(P)/\gamma^2\right\rceil$.} For any iteration {\color{black}$k\geq N$}\footnote{\textcolor{black}{For $0\leq k<N$, Assumption~\ref{ass:local} already gives $\operatorname{dist}_2(z^k,\zsset)\leq\tilde r\leq e^{1-k/N}\tilde r$.}}, suppose $c\left\lceil \frac{e^2\lambdamax^2(P)}{\gamma^2}\right\rceil \leq k < (c+1)\left\lceil \frac{e^2\lambdamax^2(P)}{\gamma^2}\right\rceil$ for a non-negative integer $c$. Thus we have
    \begin{equation*}\label{eq:eq-p-linear}
        \begin{aligned}
            &\mathrm{dist}_2(z^k,\zsset)=\frac{1}{\gamma}\gamma\mathrm{dist}_2(z^k,\zsset)\leq \frac{1}{\gamma}\mathrm{dist}_2(0,\mathcal F(z^k))\leq \frac{\sqrt{\lambdamax(P)}}{\gamma}\|z^k-z^{k-1}\|_P\\
            & \leq \frac{\sqrt{\lambdamax(P)}}{\gamma\sqrt{\left\lceil e^2\lambdamax^2(P)/ \gamma^2 \right \rceil}}\mathrm{dist}_P(z^{k-\left\lceil e^2\lambdamax^2(P)/ \gamma^2 \right \rceil},\zsset)\leq \frac{1}{\sqrt{\lambdamax(P)}}\frac{1}{e}\mathrm{dist}_P(z^{k-\left\lceil e^2\lambdamax^2(P)/ \gamma^2 \right \rceil},\zsset)\\
            &\leq \frac{1}{e}\mathrm{dist}_2(z^{k-\left\lceil e^2\lambdamax^2(P)/ \gamma^2 \right \rceil},\zsset)\leq \cdots \leq\left(\frac{1}{e}\right)^c\mathrm{dist}_2(z^{k-c\left\lceil e^2\lambdamax^2(P)/ \gamma^2 \right \rceil},\zsset)\leq \exp\pran{1-\frac{k}{\left\lceil e^2\lambdamax^2(P)/\gamma^2\right\rceil}}\tilde r \ ,
        \end{aligned}
    \end{equation*}
    where the first inequality is the definition of metric subregularity, the second and fifth inequalities follow from $P\preceq \lambdamax(P)I$, the third one uses the sublinear convergence of iterate differences~{\color{black}\cite[Theorem 2 and the discussion following (24)]{lu2025overview}}, the sixth and seventh inequalities follow from induction and the last one utilizes {\color{black}Assumption~\ref{ass:local}}.
\end{proof}

\begin{proof}[Proof of Theorem \ref{thm: local-linear}]
    By combining Corollary \ref{cor: sc-ms},  Proposition \ref{prop:local}  and Proposition \ref{thm:linear-last}, we prove Theorem \ref{thm: local-linear}.
\end{proof}

\subsection{Extension: local acceleration via Halpern and restart schemes}

In this section, we explore how local convergence of operator splitting methods can be further accelerated through a simple yet effective restart mechanism. Motivated by recent advances in restarted first-order methods for solving linear programming~\cite{applegate2023faster,lu2024restarted}, we propose a restarted variant of \eqref{eq: iteration} that enhances the local convergence, see Algorithm \ref{alg: iteration-restart}:

\begin{algorithm}[H]
\caption{Restarted \eqref{eq: iteration}}
\label{alg: iteration-restart}
\SetKwInOut{Input}{Input}
\Input{Initial point $z^{0,0}$, restart frequency $k^*$.}
initialize the outer loop counter $n\leftarrow 0$;\\
\Repeat{\upshape {\color{black}convergence}}{
  initialize the inner loop counter $k\leftarrow0$;\\
  {\color{black}set $\tilde z^{n,0}\leftarrow z^{n,0}$;}\\
  \For{$k=0,...,k^*$}{
    $z^{n,k}\gets \frac{k}{k+1}\tilde z^{n,k}+\frac{1}{k+1}z^{n,0}$;\\
    $\tilde z^{n,k+1}$ updated as in \eqref{eq: iteration} from $z^{n,k}$;
  }
  initialize the initial solution $z^{n+1,0} \leftarrow \tilde z^{n,k^*+1}$;\\
  $n\leftarrow n+1$;
}
\end{algorithm}

Algorithm \ref{alg: iteration-restart} presents our nested-loop restarted method. The algorithm begins with initialization at $z^{0,0}$. In each outer loop iteration, the Halpern variant of \eqref{eq: iteration} ({\color{black}lines 6--7}) is executed until a prescribed restart frequency is achieved. In particular, during the $k$-th inner iteration of the $n$-th outer loop, it takes a weighted
average between the current iterate $\tilde z^{n,k}$ and the initial point of this epoch $z^{n,0}$, followed by an \eqref{eq: iteration} step to obtain $\tilde z^{n,k+1}$. Once restart, the subsequent outer loop is then initialized from $\tilde z^{n,k^*+1}$ of the previous outer loop ({\color{black}line 9}).

In the following, we present the local accelerated linear convergence achieved by Algorithm \ref{alg: iteration-restart} on solving \eqref{eq: pd}.
\begin{restatable}[Local accelerated linear convergence on solving \eqref{eq: pd}]{thm}{local}\label{thm: local-linear-restart}
{\color{black}Assume the hypotheses} of Proposition \ref{prop:uqg}, i.e.,
suppose \eqref{eq: p} and \eqref{eq: d} {\color{black}admit} a {\color{black}strictly complementary} pair ${\zsol}_{0} = ({\xsol}_0,{\ysol}_0)$ and the corresponding complementary faces 
$\face^{\textcolor{black}{*}}_{{\xsol}_0}$ and $\face_{{\ssol}_0}$, where ${\ssol}_0 = c+A^{\top} {\ysol}_0$, admit quadratic facial violation. Further assume Assumption \ref{ass:full-rank} holds. Consider Algorithm \ref{alg: iteration-restart} with $k^*$ large enough. Then there are $r>0$ and $q\in (0,1)$, such that for any $z^{0,0}= (x^{0,0},y^{0,0}) \in B_r(z_{\star 0})$, the following inequality is satisfied 
 \begin{equation*}
            \operatorname{dist}_2(z^{n,0},\zsset) \leq \pran{\frac{1}{e}}^n \operatorname{dist}_2(z^{0,0},\zsset) \ .
    \end{equation*}
\end{restatable}

\begin{rem}
    Compare the local linear rate of vanilla \eqref{eq: iteration} and restarted variant Algorithm \ref{alg: iteration-restart}. Denote $\kappa=\frac{\lambdamax(P)}{\gamma}$ the condition number, where $\gamma$ is the metric subregularity parameter of $\mathcal{F}$. The requirement in the above theorem on $k^*$ in Algorithm \ref{alg: iteration-restart} is $k^* \geq\left\lceil 2e\kappa \right\rceil$. Moreover, the restart enables faster convergence with rate $O\pran{\kappa\log\frac{1}{\epsilon}}$ over linear rate $O\pran{\kappa^2\log\frac{1}{\epsilon}}$ achieved by last iterates.
\end{rem}

\subsection{Local accelerated linear convergence of restart variants}
In this section, the local behavior of restart variants is investigated, under metric subregularity. Similar {\color{black}to} the analysis in Section{\color{black}~}\ref{sec:linear-last} for last iterates, we {\color{black}first make the following local retention assumption for Algorithm~\ref{alg: iteration-restart}:
\begin{ass}[Local retention]\label{ass:local-restart}
  There is a constant $C>0$ such that 
for any optimal solution ${\zsol}_0$ satisfying strict complementarity and any small enough $r>0$, the iterates $\{z^k\}_{k=0}^\infty$ of Algorithm \ref{alg: iteration-restart} with initialization  $z^{0,0}\in B_r({\zsol}_0)$ satisfy
$z^{n,0}\in B_{\tilde r}({\zsol}_0)$ for all $n\geq 0$, where $\tilde r:=Cr$.
\end{ass}}
Parallel to Proposition \ref{prop:local}, the following proposition shows that both restarted PDHG and ADMM indeed satisfy Assumption \ref{ass:local-restart}, whose proof is deferred to Appendix \ref{app:prove-local}.
{\color{black}\begin{prop}\label{prop:local-restart}
    Suppose that \eqref{eq: p} and \eqref{eq: d} admit a strictly complementary optimal pair. Consider Algorithm~\ref{alg: iteration-restart} with PDHG satisfying $\eta\|A\|<1$, or ADMM under Assumption~\ref{ass:full-rank}. Then Assumption~\ref{ass:local-restart} holds. 

\end{prop}}


Now we establish the linear convergence of Algorithm \ref{alg: iteration-restart} under the local {\color{black}retention} assumption (Assumption~\ref{ass:local-restart}), which applies directly to both restarted PDHG and ADMM due to Proposition~\ref{prop:local-restart}.
\begin{prop}[Linear rate of restarted iterates]\label{thm:linear-last-restart}
    Consider Algorithm \ref{alg: iteration-restart} with a positive {\color{black}semidefinite} matrix $P$ for solving the primal-dual form \eqref{eq: pd} of conic program \eqref{eq: p}, with restart frequency $k^*$.  The iterates $\{z^{n,k}=(x^{n,k},y^{n,k})\}_{n=0,...,\infty}^{k=0,...,k^*}$ are obtained from this iterate update rule and the initial solution $z^{0,0}=(x^{0,0},y^{0,0})$ is chosen to satisfy Assumption \ref{ass:local-restart}, where ${\zsol}_0$, $r$ and $\tilde r:=Cr$ {\color{black}are the} corresponding constants in Assumption {\color{black}\ref{ass:local-restart}}. Suppose metric subregularity is satisfied for \eqref{eq: pd} with constant $\gamma>0$ on a region $B_{\tilde r}({\zsol}_0)\subset \mathcal{Z}$. Let restart frequency $k^* \geq\left\lceil\frac{2e{\lambdamax(P)}}{\gamma} \right\rceil$. Then it holds for any $n\ge 0$ that  
    \begin{equation*}
            \operatorname{dist}_2(z^{n,0},\zsset) \leq \pran{\frac{1}{e}}^n \operatorname{dist}_2(z^{0,0},\zsset) \ .
    \end{equation*}

\end{prop}

\begin{proof}



{\color{black}For $n=0$ the result is immediate. For} any outer iteration {\color{black}$n\geq 1$, it holds} that
\begin{equation*}
    \begin{aligned}
        &\dist_2(z^{n,0},\zsset)=\dist_2(\tilde z^{n-1,k^*+1},\zsset)\leq \frac{1}{\gamma}\dist_2(0,\mathcal F(\tilde z^{n-1,k^*+1}))\leq \frac{1}{\gamma}\|P(z^{n-1,k^*}-\tilde z^{n-1,k^*+1})\|\\
        \leq& \frac{\sqrt{\lambdamax(P)}}{\gamma}\|z^{n-1,k^*}-\tilde z^{n-1,k^*+1}\|_P\leq \frac{2\sqrt{\lambdamax(P)}}{\gamma (k^*+1)}\dist_P(z^{n-1,0},\zsset)\leq \frac{2{\lambdamax(P)}}{\gamma (k^*+1)}\dist_2(z^{n-1,0},\zsset)\\
        \leq &\frac{1}{e}\dist_2(z^{n-1,0},\zsset)
        \leq  \cdots\leq \pran{\frac{1}{e}}^n\dist_2(z^{0,0},\zsset) \ ,
    \end{aligned}
\end{equation*}
 where the first inequality is the definition of metric subregularity together with Assumption \ref{ass:local-restart}, the second one uses $P(z^{n-1,k^*}-\tilde z^{n-1,k^*+1})\in \mathcal F(\tilde z^{n-1,k^*+1})$, the third and fifth inequalities follow from $P\preceq \lambdamax(P)I$, the fourth one uses the sublinear convergence of Halpern iterates~\cite[Theorem 2.1]{lieder2021convergence}\footnote{\textcolor{black}{For ADMM, $\|\cdot\|_P$ is only a seminorm, so this theorem is not directly applicable. However, using the linear transform $h=x-\rho A^\top y$ and the fact $\|z\|_P=\rho^{-1/2}\|h\|$, 
 we can apply the theorem to the DRS state vector $h$ in \eqref {eq: def_T} of Appendix~\ref{app:prove-local} to obtain the same estimate.}}, the sixth and last inequalities follow from induction.
\end{proof}

\begin{proof}[Proof of Theorem \ref{thm: local-linear-restart}]
    Combining Corollary \ref{cor: sc-ms}, Proposition \ref{prop:local-restart},  and Proposition \ref{thm:linear-last-restart}, we prove Theorem \ref{thm: local-linear-restart}.
\end{proof}

\section{Generalization: convex composite optimization}\label{sec: generalization}
%
%
%

In this section, we consider the general convex composite optimization. Let $f\colon\mathbf{E}\to\overline{\mathbb R}$ and $g\colon\mathbf{F} \to\overline{\mathbb R}$ be proper, closed, and convex. The 
 convex composite optimization is 
\begin{equation}\label{eq: P_gen}
  \inf_{x\in\mathbf{E}}\quad f(x)+g(Ax),
  \tag{$\mathrm P_{\rm c}$}
\end{equation}
and its Fenchel dual is
\begin{equation}\label{eq: D_gen}
  \sup_{y\in\mathbf{F}}
  \{-f^*(-A^{\top}y)-g^*(y)\}.
  \tag{$\mathrm D_{\rm c}$}
\end{equation}
Note the conic model is recovered by taking $f(x)=\langle c,x\rangle+\iota_{\mathcal K}(x)$ and 
  $g(u)=\iota_{\{b\}}(u).$ 

All our previous results, as illustrated in Figure \ref{fig:proof}, continue to hold for \eqref{eq: P_gen} and \eqref{eq: D_gen}, under proper generalization. Specifically, in Section \ref{sec: gen_sc_qsv_uqg}, we generalize strict complementarity and quadratic facial violation, and show their implication of uniform quadratic growth of AL functions. In Section \ref{sec: gen_ec_ll_a},
we demonstrate the equivalence of conditions guaranteeing linear convergence of operator splitting type methods, {\color{black}local} linear convergence under them, and the acceleration. In Section \ref{subsec:qcp-application}, we show how to apply the above results to quadratic conic programs.



\subsection{Strict complementarity, quadratic subdifferential violation, and uniform quadratic growth of AL functions}\label{sec: gen_sc_qsv_uqg}


In the following, we define strong duality and generalize strict complementarity and quadratic facial violation. 

\begin{mydef}[Strong duality]\label{def: sd}
The {\color{black}problems \eqref{eq: P_gen} and \eqref{eq: D_gen}} satisfy strong duality if there is an optimal primal-dual pair $(\xsol,\ysol)\in \mathbf{E}\times \mbfF$ such that
\begin{equation}
    \pval :\,= f(\xsol) + g(A\xsol) = 
    -f^*(-A^{\top} \ysol) - g^*(\ysol) =\,: \dval. 
\end{equation}
and both $\pval$ and $\dval$ are finite. 
\end{mydef}

We motivate the general strict complementarity from the KKT condition. The KKT condition, which is equivalent to the strong duality when $f$ and $g$ are closed and convex proper functions: 
\begin{equation}\label{eq: KKT_p}
    -A^{\top} \ysol \in \partial f(\xsol),\quad \text{and}\quad \ysol \in \partial g(A\xsol),
\end{equation}
or 
\begin{equation}\label{eq: KKT_d}
    \xsol\in \partial f^*(-A^{\top} \ysol),\quad \text{and}\quad A\xsol \in \partial g^*(\ysol).
\end{equation}
Indeed, we have the following equivalence
\begin{equation}
\begin{aligned}\label{eq: opt_condition_equivalence}
f(\xsol) + g(A\xsol) = 
    -f^*(-A^{\top} \ysol) - g^*(\ysol) \in \mathbb{R}
   & \iff 
f(\xsol) + f^*(-A^{\top} \ysol) 
+g(A\xsol)+ g^*(\ysol) =0  \\
& \overset{(a)}{\iff} 
\begin{cases} 
f(\xsol) + f^*(-A^{\top} \ysol) =  -\inprd{\xsol}{A^{\top} \ysol} \\ 
 g(A\xsol)+ g^*(\ysol) = \inprd{A\xsol}{\ysol}
\end{cases} \\ 
& \overset{(b)}{\iff} 
\begin{cases} 
-A^{\top} \ysol \in \partial f(\xsol), \quad \xsol \in \partial f^*(-A^{\top} \ysol) \\
\ysol \in \partial g(A\xsol), \quad A\xsol \in \partial g^*(\ysol).
\end{cases}
\end{aligned}
\end{equation}
In the step $(a)$, we use Young's inequality: 
\begin{align}
    f(\xsol) + f^*(-A^{\top} \ysol) 
     \geq \inprd{\xsol}{-A^{\top} \ysol},  \quad \text{and}\quad  
    g(A\xsol)+ g^*(\ysol) 
     \geq \inprd{A\xsol}{\ysol}.
\end{align}
In the step $(b)$, we use the fact that $f$ and $g$ are valued in $\mathbb{R}\cup \{+\infty\}$ and both are proper, closed, and convex, and so are $f^*$ and $g^*$.

Strict complementarity requires a stronger inclusion of the last inclusion relationship in \eqref{eq: opt_condition_equivalence}:  $(\xsol, A\xsol)$ is in the  relative interior of  $\rel(\partial f^*(-A^{\top} \ysol))\times  \rel(\partial (g^*(\ysol)))$ and 
    $(\ysol,-A^{\top} \ysol)$ is in the relative interior of $\partial g(A \xsol) )\times \rel(\partial (f(\xsol)))$.  Note it indeed reduces to strict complementarity defined in Section \ref{sec: prelim} for $f(x)=\langle c,x\rangle+\iota_{\mathcal K}(x)$ and 
  $g(u)=\iota_{\{b\}}(u)$. 
\begin{mydef}[Strict complementarity: {\color{black}general} version]\label{def: sc_general}
The primal and dual programs, {\color{black}\eqref{eq: P_gen} and \eqref{eq: D_gen}, satisfy} strict complementarity if there is an optimal primal-dual pair $(\xsol,\ysol)$ such that 
\begin{subequations}\label{eq: sc_gen}
    \begin{align}
    (\xsol, A\xsol) \in \rel(\partial f^*(-A^{\top} \ysol))\times  \rel(\partial g^*(\ysol)) \label{eq: dual_sc} \\
    (\ysol,-A^{\top} \ysol) \in \rel( \partial g(A \xsol) )\times \rel(\partial f(\xsol)). \label{eq: p_sc} 
\end{align}
\end{subequations}
\end{mydef}

\begin{rem}
Apart from the conic case, we know that strict complementarity holds in a generic sense defined in \cite{drusvyatskiy2011generic} and in many structural instances \cite{ding2024sharpnesswellconditioningnonsmoothconvex}. Strict complementarity has been the core assumption in many recent works on error bounds and sensitivity \cite{ding2023strict,ding2021simplicity} and algorithm design and convergence \cite{garber2023linear,fisher2022local,garber2026randomized,ding2020spectral}.
%
\end{rem}

We generalize the previous quadratic facial violation by 
quadratic subdifferential violation: a quadratic growth condition of the function under a linear perturbation based on its subdifferential. According to \cite[Theorem 3.3]{artacho2008characterization}, the local version of this condition, restricted to a neighborhood of
$x$, is equivalent to metric subregularity of $\partial f$ at $x$ for
$v$. 
{\color{black}For a vector in the range of the subdifferential, this condition is equivalent to firm convexity relative to that vector \cite[Definition 4.1]{drusvyatskiy2018error}}. The generalized condition indeed reduces to quadratic facial violation when the function $f=\iota_K$ and $-v$ is an exposing vector. 
\begin{mydef}[Quadratic subdifferential violation] \label{def: qsv} Given a convex, closed, and proper function $f:\mathbf{E}\rightarrow \bar{\mathbb{R}}$ and {\color{black}$x\in\operatorname{dom}(\partial f):=\{z\mid\partial f(z)\neq\emptyset\}$. The subdifferential} $\partial f(x)$ satisfies quadratic subdifferential violation (QSV) relative to a vector $v\in \partial f(x)$ if for any compact set $\mathcal{X}\subset \mbfE$, there is an $\alpha>0$ such that 
\begin{equation}\label{eq: fc_inequality}
f(y) -(f(x) + \inprd{v}{y-x}) \geq  \frac{\alpha}{2} \dist^2(y,(\partial f)^{-1}(v))\quad \quad \text{for all}\; y\in \mathcal{X}.
\end{equation}
We say the function $f$ satisfies QSV if it satisfies quadratic subdifferential violation relative to any {\color{black}$v\in\bigcup_{x\in\operatorname{dom}(\partial f)}\partial f(x)$}.
\end{mydef}
We note the following useful equality of characterizing $(\partial f)^{-1}(v)$ for a $v \in \partial f(x)$ for convex proper and closed $f$: 
\begin{subequations}\label{eq: partial_f_inverse_v_all}
\begin{align}
 (\partial f)^{-1}(v) &= \partial f^*(v) \label{eq: partial_f_inverse_v} {=} \arg\min_r f(r) - \inprd{v}{r} \\ 
 -f^*(v) & = \min_r f(r) - \inprd{v}{r} = 
 f(x) - \inprd{v}{x}
 \label{eq: partial_f_inverse_v_min}
 \end{align}
\end{subequations}


\begin{rem}
    Apart from the conic examples in Section \ref{sec: Pqfv},  all strongly convex functions {\color{black}and polyhedral functions satisfy} QSV. See also \cite[Section 4]{drusvyatskiy2018error} for more examples that {\color{black}satisfy} QSV and are frequently used in applications.
\end{rem}




Recall the augmented Lagrangian defined in \eqref{eq: AL_gen}. 
We will show the following uniform quadratic growth result for augmented Lagrangians if the multipliers are optimal, near a {\color{black}strictly complementary} pair, and the QSV is satisfied.
The proof follows a similar line to the proof of Proposition \ref{prop:uqg}.

\begin{prop}\label{prop: gen_uqg}
Suppose {\color{black}\eqref{eq: P_gen} and \eqref{eq: D_gen} satisfy} the following conditions:
\begin{enumerate}
    \item Strict {\color{black}complementarity}: there is {\color{black}a} ${\zsol}_{0} = ({\xsol}_0,{\ysol}_0)$ {\color{black}satisfying} strict complementarity
    \item Quadratic subdifferential violation: the subdifferentials $\partial f({\xsol}_0)$, $\partial g (A{\xsol}_0)$, $\partial f^*(-A^{\top} {\ysol}_0)$, and $\partial g^*({\ysol}_0)$ {\color{black}satisfy} QSV relative to $-A^{\top} {\ysol}_0$, ${\ysol}_0$, ${\xsol}_0$, and $A{\xsol}_0$, respectively.
\end{enumerate}
Fix $\beta>0$. Then there is an $r>0$, such that for any compact set $\mathcal{B}\subset \mbfE^2 \times \mbfF^2$, there are constants $\alpha_1,\alpha_2$ so that the following inequality holds for any $(x,v,y,u)\in \mathcal{B}$ and  any $\zsol=(\xsol,\ysol)\in \zsset \cap B_r({\zsol}_0)$
\begin{subequations}\label{eq: AL_diff}
    \begin{align}
    L_{\beta}^P(x,u,\ysol) -  L_{\beta}^P(\xsol,A\xsol,\ysol) 
    &\geq \alpha_1 \dist^2((x,u), \mathcal{P}_\star) \label{eq: p_lag_Q}\\ 
      L_{\beta}^D(y,v,\xsol) - L_{\beta}^D(\ysol, -A^{\top} \ysol,\xsol)
    & \geq \alpha _2 \dist^2((y,v),\mathcal{D}_\star){\color{black}.} \label{eq: d_lag_Q} 
\end{align}
\end{subequations}   
where $\mathcal{P}_\star = \{(x,Ax)\mid x\in \xsset\} $ and 
$\mathcal{D}_\star = \{(y, -A^{\top} y)\mid y \in \ysset\}.$
\end{prop}

\begin{rem}
    The strict complementarity condition can be further relaxed: if any of the subdifferential sets in Definition \ref{def: sc_general} is a polyhedron, then we can remove the relative interior operator for that set. 
  The reason is that we only need a bounded linear regularity of those sets. The strict complementarity is assumed to deduce bounded linear regularity of those sets \cite[Theorem 4.6]{bauschke1999strong}. The extra layer of multiple optimal solutions, an issue raised in Remark \ref{rem: qfv}, is addressed in Lemma \ref{lem: uniform_dist} for the generalized quadratic facial violation. 
\end{rem}



\begin{proof}
Here, we only establish the quadratic growth of the primal full augmented Lagrangian, i.e., \eqref{eq: p_lag_Q}, as the argument for \eqref{eq: d_lag_Q} is identical. We first note that 
$({\xsol}_0,\ysol)$ again forms a primal-dual pair for any $(\xsol,\ysol)\in \zsset$ by optimality of ${\xsol}_0$ and $\ysol$. Thus, following \eqref{eq: opt_condition_equivalence}, we have $(\ysol,-A^{\top} \ysol) \in \partial g(A {\xsol}_0 )\times \partial f({\xsol}_0)$. Hence, for all small enough $r$ and any $(\xsol,\ysol)\in \zsset \cap B_r({\color{black}{\zsol}_0})$, we have 
\begin{equation}\label{eq: sc_close_opt_pair}
    (\ysol,-A^{\top} \ysol) \in \rel(\partial g(A {\xsol}_0) )\times \rel(\partial f({\xsol}_0)).
\end{equation}

%


\paragraph{Rewriting Primal Lagrangian as sum of KKT violations} 
We rewrite $L_{\beta}^P(x,{\color{black}u},\ysol) -  L_{\beta}^P(\xsol,A\xsol,\ysol)$ as the following:
\begin{equation}
\begin{aligned} \label{eq: L_D_xs_linear_infeas_form}
& L_{\beta}^P(x,u,\ysol) -  L_{\beta}^P(\xsol,A\xsol,\ysol)\\ 
  \overset{(a)}{=} &f(x) -f({\xsol}_0) +g(u) - g(A{\xsol}_0) +
 \inprd{\ysol}{Ax-u}+\frac{1}{2\beta} \norm{Ax-u}^2\\
  = & \underbrace{f(x) -f({\xsol}_0)- \inprd{-A^{\top} \ysol}{x-{\xsol}_0}}_{I}  +
  \underbrace{g(u) -g(A{\xsol}_0)- \inprd{\ysol}{u-A{\xsol}_0}}_{II} + \frac{1}{2\beta} \norm{Ax-u}^2 \\ 
\end{aligned}
\end{equation}
Here, in the step $(a)$, we use the equality $L_\beta^P(\xsol,A\xsol,\ysol) = L_\beta^P ({\xsol}_0,A{\xsol}_0,\ysol)$ due to the optimality of $\xsol$ and ${\xsol}_0$. Note that the three terms $(I)$, $(II)$, and $\frac{1}{2\beta}\norm{Ax-u}^2$ are nonnegative, thanks to the KKT condition \eqref{eq: KKT_p}. Moreover, note that the terms $I$, $II$, and the last term correspond to the KKT conditions: $x\in \partial f^*(-A^{\top} \ysol)$, $u \in \partial g^*(\ysol)$, and $Ax =u$.   

\paragraph{Bounded linear regularity} Define the sets $C = \{(x,Ax)\mid x\in \mbfE\} \subset \mbfE \times \mbfF$, $D = \mbfE \times \partial g^*({\ysol}_0)$, and $E = 
\partial f^*(-A^{\top} {\ysol}_0)\times \mbfF$. 
Note that the intersection of $C$, $D$, and $E$ is the set of optimal primal solutions of {\color{black}\eqref{eq: P_gen}}. Thanks to strict complementarity, we know that 
\[
({\xsol}_0, A{\xsol}_0 ) \in C\cap \rel(D) \cap \rel(E)\not=\emptyset 
\]
Using the above relationship and $C$ is an affine space, we can activate bounded linear regularity \cite[Corollary 5]{bauschke1999strong}, which states that 
for any bounded set $\mathcal{B}$, there is {\color{black}a} constant $c_1>0$ such that
$\dist((x,u),C\cap D\cap E) \leq c_1( \dist((x,u),C) + \dist((x,u),D)+\dist((x,u),E)$ for any $(x,u) \in \mathcal{B}$.
By recalling the definition of $C$, $D$, $E$ we have that for any compact set $B_2 \subset \mbfE\times \mbfF$, there is {\color{black}a} $c_2>0$ such that 
\begin{equation}\label{eq: blr_dual_side}
\dist((x,u), \mathcal{P}_\star) \leq 
c_2 (\norm{Ax-u}+ \dist(x,\partial f^*(-A^{\top} {\ysol}_0)) + 
\dist(u,\partial g^*({\ysol}_0))).
\end{equation}

\paragraph{Linking via uniform QSV} Next, we link \eqref{eq: L_D_xs_linear_infeas_form} and \eqref{eq: blr_dual_side} via the 
uniform firm convexity established in Lemma \ref{lem: uniform_dist}. First, we know from Lemma \ref{lem: uniform_dist} and \eqref{eq: sc_close_opt_pair} that there is an $r_0$ such that for any compact set $\mathcal{B}\subset \mbfE \times \mbfF$, the following holds for some $\alpha_0>0$: for all $(x,u)\in \mathcal{B}$, we have 
\begin{subequations}\label{eq: f_g_qg}
\begin{align}
{\color{black}f(x) - f(}{{\color{black}\xsol}}{\color{black}_0) -\inprd{-A^{\top} \ysol}{x-{\xsol}_0}} & {\color{black}\geq \alpha_0 \dist^2(x,(\partial f)^{-1}(-A^{\top}{\ysol}_0))} \\
{\color{black}g(u)- g(A}{{\color{black}\xsol}}{\color{black}_0) - \inprd{\ysol}{u-A{\xsol}_0}}&{\color{black}\geq \alpha_0\dist^2(u,(\partial g)^{-1}(}{{\color{black}\ysol}}{\color{black}_0))} 
\end{align}
\end{subequations}
Hence, for the term $(I)$ and $(II)$ in  \eqref{eq: L_D_xs_linear_infeas_form}, using \eqref{eq: f_g_qg}, we have 
\begin{equation}
    {\color{black}\begin{aligned}
  & L_{\beta}^P(x,u,\ysol) -  L_{\beta}^P(\xsol,A\xsol,\ysol)
  \geq 
  \alpha_0 (\dist^2(x,(\partial f)^{-1}(-A^{\top}{\ysol}_0)) + \dist^2(u,(\partial g)^{-1}({\ysol}_0))) + \frac{1}{2\beta} \norm{Ax-u}^2.\\
    \end{aligned}}
\end{equation}
Combining the above with \eqref{eq: blr_dual_side}, and {\color{black}utilizing} the fact \eqref{eq: partial_f_inverse_v_all}, we see \eqref{eq: p_lag_Q} holds with some $\alpha_1>0$.
\end{proof}

\begin{rem}
A key property of Proposition \ref{prop: gen_uqg} is that 
$\alpha_1$ and $\alpha_2$ are independent of $(\xsol,\ysol)$.
Without such a requirement, we may pick a small enough radius $r$ of the ball $B_r({\zsol}_0)$ so that strict complementarity {\color{black}holds} for any $(\xsol,\ysol)\in \mathcal{Z}_\star \cap {\color{black}B_r({\zsol}_0)}$. We may then repeat the above argument, replacing every $({\xsol}_0,{\ysol}_0)$ by $(\xsol,\ysol)$ to establish quadratic growth of $L^P_\beta$ and $L^D_\beta$ for this particular pair. Though $\alpha_1$ and $\alpha_2$ do depend on $(\xsol,\ysol)$ in this case.
\end{rem}
\begin{lem}\label{lem: inner_p_bound_diff_x_in_rel_face}
 Consider a {\color{black}proper closed convex} function $f:\mbfE \rightarrow\bar{\mathbb{R}}$ and fix an $x\in \mbfE$ such that $\partial f(x)\not= \emptyset$. Suppose {\color{black}$v_0\in\partial f(x)$ and either $v_0\in\rel(\partial f(x))$ or $\partial f(x)$ is} polyhedral. Then there is an $r>0$ such that for any $v\in \partial f(x)\cap B_r(v_0)$, we have 
\begin{equation}
  \label{eq: inner_p_bound_diff_x_in_rel_face_general}
f(y) -f(x) - \inprd{v}{y-x} \geq \frac{1}{2} \left(
f(y) -f(x)  - \inprd{v_0}{y-x}\right) \quad \quad \text{for all}\;y \in \mbfE.
  \end{equation}
\end{lem}
\begin{proof}
For any $v\in \partial f(x)$, we know the inequality \eqref{eq: inner_p_bound_diff_x_in_rel_face_general} is equivalent to 
\begin{equation}\label{eq: equivalent_inequality_relating_subgradient_inequality}
    f(y) -f(x)- \inprd{2v-v_0}{y-x} \geq 0\quad \quad \text{for all}\;y \in \mbfE.
\end{equation}
Hence, it is enough to show \eqref{eq: inner_p_bound_diff_x_in_rel_face_general} if we can find an $r>0$ such that for any $v\in \partial f(x)\cap {\color{black}B_r(v_0)}$, we have 
$2v-v_0 \in \partial f(x)$.  We assume $\partial f(x)$ contains multiple points in the following, as {\color{black}$2v-v_0\in\partial f(x)$} is immediate if $\partial f(x)$ is a singleton.

Consider the first case $v_0\in \rel(\partial f(x))$. Then there is an $r_0>0$ such that {\color{black}$\aff(\partial f(x)) \cap {\color{black}B_{r_0}(v_0)}\subset \partial f(x)$, where $\aff(\partial f(x))$} is the smallest affine space containing $\partial f(x)$. Take $r= \frac{r_0}{4}$ and consider  any $v\in \partial f(x)\cap {\color{black}B_r(v_0)}$. Because $2v-v_0 \in \aff(\partial f(x))$ and $\norm{2v-v_0 - v_0}=2\norm{v-v_0}\leq \frac{r_0}{2}$, we know 
$2v-v_0 \in \partial f(x)$. 

Next, consider the second case that $\partial f(x)$ is polyhedral. It is straightforward to show that such a radius $r$ exists by considering the active and inactive inequalities of the polyhedral representation. Indeed, suppose $\partial f(x) = \{y \in \mbfE\mid By\geq b\}$ for some matrix $B$ and vector {\color{black}$b$}. Let $I_> = \{i \mid (Bv_0)_i>b_i\}$ and $I_= =\{i \mid (Bv_0)_i=b_i\}$. Suppose $I_>\not=\emptyset$ for now and we have $\delta = \min_{i\in I_>} (Bv_0)_i -b_i>0$. Then we see by choosing a radius $r$ small enough, we will have that for any $v\in \partial f(x) \cap {\color{black}B_r(v_0)}$, $(Bv)_i\geq b_i+\frac{\delta}{2}$ for all $i\in I_>$, $(Bv)_i\geq b_i$ for all $i\in I_=$, and $|(B(v-v_0))_i|\leq \frac{\delta}{4}$ for all $i$. Then  for any $i\in I_>$, we have
$(B(2v-v_0))_i = (Bv)_i + (B(v-v_0))_i\geq b_i + \frac{\delta}{2}-\frac{\delta}{4} \geq b_i$. And for any $i\in I_=$, we have $(B(2v-v_0))_i = 
2(Bv)_i-(Bv_0)_i = 2(Bv)_i - b_i \geq 2b_i-b_i =b_i$. If $I_>=\emptyset$, then our proof for the $I_=$ case still works, and there is nothing more to prove.  Our proof is complete.
\end{proof}

\begin{lem}[Uniform QSV]\label{lem: uniform_dist}
    {\color{black}Assume the hypotheses of} Lemma \ref{lem: inner_p_bound_diff_x_in_rel_face}. Further suppose $f$ is firmly convex relative to $v_0$. Then, 
    there is an $r>0$ such that, for any compact set $\mathcal{B}$, there is an $\alpha_0>0$ with the property that for all $v \in \partial f(x) \cap {\color{black}B_r(v_0)}$, we have 
\begin{equation}
  \label{eq: inner_p_bound_diff_x_in_rel_face_strengthened_general}
f(y) -f(x) - \inprd{v}{y-x} \geq \frac{\alpha_0}{2} \dist^2(y,(\partial f)^{-1}(v_0))  \quad \quad \text{for all}\;y \in \mathcal{B}.
  \end{equation}
\end{lem}
\begin{proof}
    Since $\partial f(x)$ satisfies QSV relative to $v_0$, we know that for any compact set $\mathcal{B}$, there is an $\alpha>0$ such that for all $y\in \mathcal{B}$ 
    \begin{equation}
        f(y) -f(x) - \inprd{v_0}{y-x} \geq \frac{\alpha}{2} \dist^2 (y, (\partial f)^{-1}(v_0)). 
    \end{equation}
    Using Lemma \ref{lem: inner_p_bound_diff_x_in_rel_face}, we know that 
    \eqref{eq: inner_p_bound_diff_x_in_rel_face_strengthened_general} holds with $\alpha_0 = \frac{\alpha}{2}$.
\end{proof}

\subsection{Equivalent conditions, local linear convergence, and acceleration}\label{sec: gen_ec_ll_a}
In this section, we discuss how our results in Section \ref{sec:equiv} and \ref{sec:linear} regarding equivalent conditions for local linear convergence, local linear convergence under them, and acceleration, generalize to the convex composite setting. 

\paragraph{Equivalence conditions, and local linear convergence and acceleration under them} Thanks to our notation and arguments in Section \ref{sec:equiv} and \ref{sec:linear}, it is {\color{black}evident} that the definitions, propositions, theorems, and algorithms do not rely on the particular conic structure and they generalize seamlessly to the convex composite setting without any modification. In particular, the equivalence result on the conditions for local linear convergence, Theorem \ref{thm:equivalence}, the local linear convergence under them, Proposition {\color{black}\ref{thm:linear-last}}, and the acceleration, Proposition {\color{black}\ref{thm:linear-last-restart}}, {\color{black}continue} to hold for \eqref{eq: P_gen} without any change in their statements and proofs.

\paragraph{Implications of strict complementarity and QSV} It is also evident that under strict complementarity and the quadratic subdifferential violation, the three equivalent conditions continue to hold, as well as the local linear convergence and acceleration. We record these results here. The proof follows the same line as the {\color{black}proofs} of Corollary \ref{cor: sc-ms}, Theorem \ref{thm: local-linear}, and Theorem \ref{thm: local-linear-restart}. 

\begin{restatable}{thm}{localcomposite}\label{thm: uqg_consequence_gen}
{\color{black}Assume Assumption~}\ref{ass:full-rank} and the {\color{black}hypotheses} of Proposition \ref{prop: gen_uqg}. Then the following three statements hold:
\begin{enumerate}
\item Equivalent conditions hold: the three conditions described in Theorem \ref{thm:equivalence}  for \eqref{eq: P_gen} and \eqref{eq: D_gen}  hold.
\item Local linear convergence: Consider $z^k=(x^k,y^k)$ from either ADMM or PDHG. Then there are $r,\tilde{r}>0$ and $q\in (0,1)$, such that for any $z^0= (x^0,y^0) \in B_r(z_{\star 0})$, the following inequality is satisfied 
  \begin{equation*}
    \mathrm{dist}_2(z^k,\zsset) \leq q^k \tilde{r}.
\end{equation*}
    \item Acceleration:
consider Algorithm \ref{alg: iteration-restart} {\color{black}with $k^*\geq\left\lceil 2e\lambdamax(P)/\gamma\right\rceil$, where $\gamma>0$ is a local metric subregularity constant}. Then there are $r>0$ and $q\in (0,1)$, such that for any $z^{0,0}= (x^{0,0},y^{0,0}) \in B_r(z_{\star 0})$, the following inequality is satisfied 
 \begin{equation*}
            \operatorname{dist}_2(z^{n,0},\zsset) \leq \pran{\frac{1}{e}}^n \operatorname{dist}_2(z^{0,0},\zsset) \ .
    \end{equation*}
\end{enumerate}
\end{restatable}

\subsection{Application: convex quadratic conic programs}
\label{subsec:qcp-application}

In this section, we consider {\color{black}a} convex
quadratic conic program, i.e., Problem \eqref{eq: p} with an additional positive {\color{black}semidefinite} quadratic form $\frac{1}{2}\inprd{x}{Qx}$, where the linear map $Q:\mathbf{E}\rightarrow\mathbf{E}$ is self-adjoint and positive semidefinite,
\begin{equation}  \tag{$\mathrm{P}_{Q}$}
  \label{eq:qcp-primal}
\begin{aligned}
\min & \  \frac{1}{2} \inprd{x}{Qx}+\inprd{c}{x}
        \quad \mathrm{s.t.}\ \ Ax=b,\quad x\in \mathcal{K} \ .
\end{aligned}
\end{equation}
It has the composite form $\inf_x\{f_Q(x)+g(Ax)\}$ with
\begin{equation}
  f_Q(x)
  :=\frac{1}{2}\langle x,Qx\rangle+\langle c,x\rangle
    +\iota_{\mathcal{K}}(x),
  \qquad
  g(u):=\iota_{\{b\}}(u).
  \label{eq:qcp-fg}
\end{equation}
In the following, assuming strong duality (Definition \ref{def: sd}) throughout, we shall investigate the general strict complementarity and QSV defined in Section \ref{sec: gen_sc_qsv_uqg} in concrete geometric terms. We further demonstrate that the desired results in Theorem \ref{thm: uqg_consequence_gen} hold under the conditions proposed in \cite{li2026gpu} using the framework in prior subsections by verifying strict complementarity and QSV. In particular, our route complements the argument in concurrent work~\cite{li2026gpu} in verifying the quadratic growth of the smoothed gap. The work~\cite{li2026gpu} uses a conic argument with concepts such as variable-lifting, the weighted projection, slack-regularity, calmness of normal cones, and {\color{black}$C^2$}-cone reducibility. In contrast, our proof proceeds by directly verifying the composite strict-complementarity and QSV conditions established herein.

\paragraph{Strict complementarity} We start with strict complementarity. Let $(x_{\star 0},y_{\star 0})$ be a pair of optimal {\color{black}solutions}. We shall investigate the subdifferentials {\color{black}$\partial f_Q({\xsol}_0)$}, $\partial g (A{\xsol}_0)$, {\color{black}$\partial f_Q^*(-A^{\top} {\ysol}_0)$}, and $\partial g^*({\ysol}_0)$. Since {\color{black}$g = \iota_{\{b\}}$, we have} $g^*(y) =\langle b,y\rangle$.
Thus, we have 
\begin{equation}
\partial g^*(y)=\{b\} = \rel(\partial g^*({\ysol}_0))\quad \text{and}\quad \partial g(A{\xsol}_{0})=\mathbf{F} = \rel(\partial g(A{\xsol}_{0})). 
\end{equation}
It is then clear {\color{black}that} strict complementarity for $g$ and $g^*$ holds automatically. 

Next, we examine the strict complementarity associated with $\partial f_Q({\xsol}_0)$. 
Define the slack $s_{\star 0}:=Q{\xsol}_0+c+A^{\top}y_{\star 0}$
 and recall complementary faces $\mathcal{C}_{{\ssol}_0} = \mathcal{K}\cap {\ssol}_0^\perp$ and $\mathcal{C}_{{\xsol}_0} = \mathcal{K}^*\cap {\xsol}_0^\perp$. For the case of $\partial f_Q({\xsol}_0)$, using the KKT condition \eqref{eq: KKT_p} (which follows from our strong duality), we have  
 \begin{equation}\label{eq: partial_f_Q}
-A^\top {\ysol}_0 \in \partial f_Q({\xsol}_0) = Q{\xsol}_0 + c + N_{\mathcal{K}}({\xsol}_0) \iff 
 {\ssol}_0 \in - N_{\mathcal{K}}({\xsol}_0) 
 \overset{(a)}{\iff} {\ssol}_0 \in \mathcal{C}_{{\xsol}_0},
 \end{equation}
 where the step $(a)$ uses the simple fact that $s\in -N_{\mathcal{K}}(x) \iff s \in \mathcal{K}^*\cap x^\perp$, which follows from the definition of the normal cone and the nonnegative homogeneity of the cone. Thus, the strict complementarity of $-A^\top {\ysol}_0 \in \partial f({\xsol}_0) $ in geometric {\color{black}terms is simply ${\ssol}_0$} lies in the relative interior of $\mathcal{C}_{{\xsol}_0}$, i.e.,
 \begin{equation}\label{eq: p_sc_qcp}
     {\ssol}_0 \in \rel(\mathcal{C}_{{\xsol}_0}).
 \end{equation}
Finally, we examine the case for $\partial f^*_Q(-A^{\top} {\ysol}_0)$. Denote the null space of $Q$ by $\ker Q = \{x \mid Qx =0\}$. We claim that 
\begin{equation}\label{eq: partial_f_star_Q}
    \partial f^*_Q(-A^{\top} {\ysol}_0)=({\xsol}_0+\mbox{ker}(Q))\cap \mathcal{K}\cap {{\ssol}_0}^\perp = ({\xsol}_0+ \mbox{ker}(Q)) \cap \mathcal{C}_{{\ssol}_0}
\end{equation}
Indeed, 
using \eqref{eq: KKT_p} ($-A^{\top} {\ysol}_0 \in \partial f_Q({\xsol }_0)$) and \eqref{eq: partial_f_inverse_v_all} (the characterization of $x \in \partial f^*_Q(-A^{\top} {\ysol}_0)$),  we have $ x \in \partial f^*_Q(-A^{\top} {\ysol}_0)$ if and only if 
\begin{equation}
    \begin{aligned}      
& \inprd{-A^\top {\ysol}_0}{{\xsol}_0} - f_Q({\xsol}_0)=
f^*_Q(-A^\top {\ysol}_0) = \inprd{-A^\top {\ysol}_0}{x} - f_Q(x)\\ 
\overset{(a)}{\implies} & 0 = \frac{1}{2} \inprd{(x-{\xsol}_0)}{Q(x-{\xsol}_0)}
+ \inprd{{\ssol}_0}{x},
    \end{aligned}
\end{equation}
where we use $\inprd{{\ssol}_0}{{\xsol}_0}=0$ in the step $(a)$ (as ${\ssol}_0\in \mathcal{K}^* \cap {{\xsol}_0}^\perp$). Thus, the equality \eqref{eq: partial_f_star_Q} holds. Hence, the strict complementarity for $\partial f^*_Q(-A^{\top} {\ysol}_0)$ in {\color{black}geometric terms} is 
\begin{equation}\label{eq: d_sc_qcp}
    {\xsol}_0 \in \rel( ({\xsol}_0+ \mbox{ker}(Q)) \cap \mathcal{C}_{{\ssol}_0}). 
\end{equation}
Note {\color{black}that} $\rel(({\xsol}_0+ \mbox{ker}(Q))\cap \mathcal{C}_{{\ssol}_0})$ is in general different from $ \mbox{ker}(Q) \cap \rel(\mathcal{C}_{{\ssol}_0}).$
The two strict complementarity type conditions, which we will call \emph{facial strict complementarity}, imposed in 
\cite{li2026gpu} are
\begin{equation}\label{eq: sc_facial_qcp}\tag{F-SC}
  \quad x_{\star 0}\in\operatorname{ri}\mathcal{C}_{{\ssol}_0}
  \quad\text{and}\quad
 s_{\star 0}\in\operatorname{ri}\mathcal{C}_{{\xsol}_0}.
\end{equation}
The left facial condition, $x_{\star 0}\in\operatorname{ri}\mathcal{C}_{{\ssol}_0}$ should be distinguished from the 
strict complementarity condition \eqref{eq: d_sc_qcp}, ${\xsol}_0 \in \rel( ({\xsol}_0+ \mbox{ker}(Q)) \cap \mathcal{C}_{{\ssol}_0})$. It is also immediate that facial strict complementarity \eqref{eq: sc_facial_qcp} implies the strict complementarity using the relative interior calculus. We summarize the relationship in the following lemma.

\begin{lem}[Facial strict complementarity implies strict complementarity]
\label{lem:qcp-sc-comparison}
If both conditions in~\eqref{eq: sc_facial_qcp} hold, then the composite
representation~\eqref{eq:qcp-fg} satisfies the strict complementarity
condition. 
\end{lem}

\paragraph{Quadratic facial violation} Next, we investigate QSV of $g$, $g^*$, $f_Q$ and $f_Q^*$.  Since $g=\iota_{\{b\}}$, one has
$(\partial g)^{-1}(y_{\star 0})= \arg\min_r g(r) - \inprd{{\ysol}_0}{r} = \{b\}$. Thus, the QSV inequality \eqref{eq: fc_inequality} holds as it is exact at $b$ and has an infinite left-hand side away from $b$. Moreover,
$g^*(y)=\langle b,y\rangle$ and
$(\partial g^*)^{-1}(b)= \partial g(b) = \mathbf{F}$, so the corresponding distance term in the QSV inequality \eqref{eq: fc_inequality} is identically zero and QSV holds trivially. To this end, let us consider the following QFV conditions, which hold for any face and any defining vector considered in \cite{li2026gpu}: on every compact set $B$,
there are constants $\kappa_p,\kappa_d>0$ such that
\begin{subequations}
\label{eq:qcp-qfv}
\begin{align}
  \operatorname{dist}^2(x,\mathcal{C}_{{\ssol}_0})
  &\leq \kappa_p\langle s_{\star 0},x\rangle,
  &&x\in\mathcal{K}\cap B,
  \label{eq:qcp-primal-qfv}\\
  \operatorname{dist}^2(s,\mathcal{C}_{{\xsol}_0})
  &\leq \kappa_d\langle x_{\star 0},s\rangle,
  &&s\in\mathcal{K}^* \cap B.
  \label{eq:qcp-dual-qfv}
\end{align}
\end{subequations}
We now investigate QSV for $f_Q$ and $f_Q^*$.

\underline{QSV of $\partial f_Q({\xsol}_0)$ relative to {\color{black}$-A^\top {\ysol}_0$}.} The {\color{black}left-hand} side of the QSV inequality \eqref{eq: fc_inequality} of $\partial f_Q({\xsol}_0)$ relative to $-A^\top {\ysol}_0$ reads as (using $\inprd{{\ssol}_0}{{\xsol}_0}=0$)
\begin{equation}\label{eq: qsv_qcp_f_Q}
f_Q(x)-f_Q(x_{\star 0})
  -\left\langle -A^{\top}y_{\star 0},x-x_{\star 0}\right\rangle  
= \frac{1}{2} \inprd{(x-{\xsol}_0)}{Q(x-{\xsol}_0)} + \inprd{{\ssol}_0}{x} + \iota_{\mathcal{K}}(x). 
\end{equation} 
Note that all the terms are nonnegative and are zero if and only if   {\color{black}$x\in  ({\xsol}_0+ \mbox{ker}(Q) )\cap \mathcal{C}_{{\ssol}_0} \overset{\eqref{eq: partial_f_star_Q}}{=}\partial f^*_Q(-A^{\top} {\ysol}_0) = (\partial f_Q)^{-1} (-A^{\top} {\ysol}_0)$}. Thus, assuming the QFV \eqref{eq:qcp-primal-qfv}, we see \eqref{eq: fc_inequality} holds once  the bounded linear regularity of the closed convex set
$\mathcal{C}_{{\ssol}_0}$ and the affine set $x_{\star 0}+\ker Q$ holds, i.e., for any 
compact set $B$, there is $c_B>0$ such that for any $x\in B$
\begin{align}
  &\operatorname{dist}^2\!\left(
    x,\mathcal{C}_{{\ssol}_0}\cap(x_{\star 0}+\ker Q)\right)
\leq c_B\left[
    \operatorname{dist}^2(x,\mathcal{C}_{{\ssol}_0})
    +\operatorname{dist}^2(x,x_{\star 0}+\ker Q)
  \right].
  \label{eq:qcp-primal-linear-regularity}
\end{align}
Indeed, if $Q=0$, then $\ker Q = \mathbf{E}$ and 
$\operatorname{dist}^2(x,x_{\star 0}+\ker Q) = \frac{1}{2} \inprd{(x-{\xsol}_0)}{Q(x-{\xsol}_0)} =0 $ for any $x$. Hence, the QSV inequality 
\eqref{eq: fc_inequality} holds by linking \eqref{eq: qsv_qcp_f_Q} and \eqref{eq:qcp-primal-linear-regularity}. Otherwise, we have $Q\neq 0$. Then, letting $\lambda_{\min}^+(Q)$ be the smallest nonzero eigenvalue of $Q$, we have 
\begin{equation}
  \operatorname{dist}^2(x,x_{\star 0}+\ker Q)
  \leq
  \frac{1}{\lambda_{\min}^+(Q)}
  \langle x-x_{\star 0},Q(x-x_{\star 0})\rangle;
  \label{eq:qcp-quadratic-slice}
\end{equation}
 Combining
\eqref{eq: qsv_qcp_f_Q},
\eqref{eq:qcp-primal-qfv},
\eqref{eq:qcp-quadratic-slice}, 
and \eqref{eq:qcp-primal-linear-regularity}
gives the QSV inequality 
\eqref{eq: fc_inequality} of $\partial f_Q({\xsol}_0)$ relative to {\color{black}$-A^\top {\ysol}_0$}. Since  
the bounded linear regularity holds if ${\xsol}_0\in \rel(\mathcal{C}_{{\ssol}_0})\cap (x_{\star 0}+\ker Q)$ which is implied by the facial strict complementarity \eqref{eq: sc_facial_qcp}, the QSV of $\partial f_Q({\xsol}_0)$ relative to {\color{black}$-A^\top {\ysol}_0$} holds under 
\eqref{eq:qcp-primal-qfv} and \eqref{eq: sc_facial_qcp}.
\underline{QSV for $\partial f_Q^*(-A^\top {\ysol}_0)$ relative to ${\xsol}_0$.} 
Let $s = c+Q^\top {\xsol}_0 -v$ and 
$\phi(x) = \frac{1}{2}\inprd{x-{\xsol}_0}{Q(x-{\xsol}_0)}$ for $x\in \mathbf{E}$. Denote
the left-hand side of the QSV inequality \eqref{eq: fc_inequality}  as $\Delta_D(v) = f_Q^*(v)- \left(f_Q^*(-A^\top {\ysol}_0) + \inprd{{\xsol}_0}{v- (-A^\top {\ysol}_0)}\right)$. Then, we have 
\begin{equation}
\begin{aligned}
\Delta_D(v)
\overset{(a)}{=}& f_Q^*(v) + f_Q({\xsol}_0) -\inprd{{\xsol}_0}{v}
=  \sup_{x} \inprd{v}{x} - f_Q(x) + f_Q({\xsol}_0)-\inprd{{\xsol}_0}{v} \\
\overset{(b)}{=} & \sup_{x} \inprd{-s}{x} - \phi(x) -\iota_{\mathcal{K}}(x) + \inprd{s}{{\xsol}_0} 
=  (\phi + \iota_{\mathcal{K}})^*(-s) + \inprd{s}{{\xsol}_0}\\
\overset{(c)}{=} &  \min_{ u_1 + u_2= -s} \phi^*(u_1) + \iota_{\mathcal{K}}^*(u_2)+ \inprd{s}{{\xsol}_0},
  \label{eq:qcp-dual-exact-gap}
\end{aligned}
\end{equation}
where the step $(a)$ is because $f_Q^*(-A^\top {\ysol}_0) = 
-f_Q({\xsol}_0) + \inprd{-A^\top {\ysol}_0}{{\xsol}_0}$ due to
\eqref{eq: partial_f_inverse_v_min} and $-A^\top {\ysol}_0 \in \partial f_Q({\xsol}_0)$; the step $(b)$ is due to some algebra manipulation; and the step $(c)$ is the {\color{black}conjugate sum rule} \cite[Theorem~16.4]{rockafellar1970convex} and $\phi$ has the whole domain. Since 
$\phi^*(u_1) = \inprd{u_1}{{\xsol}_0} + \frac{1}{2} \inprd{u_1}{Q^{\dagger} u_1} + \iota_{\range (Q)}(u_1)$ where $Q^\dagger$ is the Moore-Penrose inverse, and $\iota_{\mathcal{K}}^*(u_2) = \iota_{\mathcal{K}^*}(-u_2)$, after eliminating $u_2$ using $u_1+u_2 = -s$, we see   \eqref{eq:qcp-dual-exact-gap} becomes 
\begin{equation}
\begin{aligned}
\Delta_D(v)
= &  \min_{u_1 \in \range(Q), s+u_1\in \mathcal{K}^*} \inprd{u_1 +s}{{\xsol}_0} + \frac{1}{2}\inprd{u_1}{Q^\dagger u_1}.
  \label{eq:qcp-dual-exact-gap-conti}
\end{aligned}
\end{equation}
Hence, if  $Q=0$, then QSV for $\partial f_Q^*(-A^\top {\ysol}_0)$ relative to ${\xsol}_0$ reduces to the QFV \eqref{eq:qcp-dual-qfv} and we are done. We assume $Q\not=0$ in the following. When $Q\not=0$ and $f_Q^*(v)$ is finite, the last minimization in \eqref{eq:qcp-dual-exact-gap-conti} is strongly convex in $u_1 \in \range(Q)$. Hence, denoting the {\color{black}largest} eigenvalue of $Q$ as $\|Q\|$, there is a unique minimizer $u_1(s) \in \range(Q)$ such that
\begin{equation} \label{eq:qcp-dual-exact-gap-conti-2}
\begin{aligned}
   \Delta_D(v)
    &= 
    \inprd{u_1(s)+s}{{\xsol}_0} + \frac{1}{2}\inprd{u_1(s)}{Q^\dagger u_1(s)} 
    &\geq     \inprd{u_1(s)+s}{{\xsol}_0} + 
    \frac{1}{2\|Q\|}\|u_1(s)\|^2.
\end{aligned} 
\end{equation}
When $\Delta_D(v)\leq 1$, its minimizing
$u_1(s)$ is bounded by $\|u_1(s)\|_2 \leq \sqrt{2 \|Q\|}$ using \eqref{eq:qcp-dual-exact-gap-conti-2}. Fix a compact set $B_v$, then the vector $s = c+Q{\xsol}_0-v$ lives in another compact set $c+Q{\xsol}_0-B_v$. Hence, the QFV \eqref{eq:qcp-dual-qfv} with $B = 
c+Q{\xsol}_0-B_v + B_{\sqrt{2\|Q\|}}(0)$ can be applied to $u_1(s)+s \in B\cap \mathcal{K}^*$. Since the right-hand side of \eqref{eq: fc_inequality} is $\dist(v, (\partial f_Q^*)^{-1}({\xsol}_0)) \overset{\eqref{eq: partial_f_inverse_v}}{=}  \dist(v, \partial f_Q({\xsol}_0)) 
\overset{\eqref{eq: partial_f_Q}}{=} \dist (v, 
c+ Q{\xsol}_0 - \mathcal{C}_{{\xsol}_0}) = 
\dist(s, \mathcal{C}_{{\xsol}_0})$, we see 
\begin{equation}\label{eq: u_inequality}
\dist(v, (\partial f_Q^*)^{-1}({\xsol}_0)) =\dist(s,\mathcal{C}_{{\xsol}_0}) \leq \dist(s+u_1(s),\mathcal{C}_{{\xsol}_0}) + \|u_1(s)\|_2.
\end{equation}
The QSV {\color{black}inequality} \eqref{eq: fc_inequality} thus holds when $\Delta_D(v)\leq 1$ by combining the QFV \eqref{eq:qcp-dual-qfv} with {\color{black}$B = 
c+Q{\xsol}_0-B_v + B_{\sqrt{2\|Q\|}}(0)$}, \eqref{eq:qcp-dual-exact-gap-conti-2}, and \eqref{eq: u_inequality}. 

When the left-hand side in~\eqref{eq:qcp-dual-exact-gap-conti}  is larger than one, compactness of $B_v$ bounds
$\operatorname{dist}(v, (\partial f_Q^*)^{-1}({\xsol}_0))$, so the desired inequality follows after decreasing the QSV coefficient $\alpha$ in \eqref{eq: fc_inequality}. {\color{black}Hence, the QSV inequality \eqref{eq: fc_inequality} holds with a positive coefficient $\alpha$.} Note we only use QFV \eqref{eq:qcp-dual-qfv} here.

We summarize the above finding in the following lemma.
\begin{lem}[QSV for the quadratic conic specialization]
\label{lemma:qcp-qsv-addon}
Suppose the left facial strict complementarity condition \eqref{eq: sc_facial_qcp} and the two QFV conditions in~\eqref{eq:qcp-qfv}
hold at $(x_{\star 0},y_{\star 0})$. Then $\partial f_Q(x_{\star 0}),\;\partial f_Q^*(-A^{\top}y_{\star 0}), \;\partial g(b)$, and $\partial g^*(y_{\star 0})$ satisfy QSV relative to
  $-A^{\top}y_{\star 0}$, 
  $x_{\star 0}$, $y_{\star 0}$, and $b$, respectively.
\end{lem}

With the above lemmas, we conclude this section with the following corollary. 
\begin{cor}[Quadratic conic programs as an application of Section~6]
\label{cor:qcp-section6-application}
Under the facial strict complementarity conditions \eqref{eq: sc_facial_qcp}, the two QFV conditions in~\eqref{eq:qcp-qfv}, and that $A$ is surjective, the
quadratic conic program~\eqref{eq:qcp-primal} satisfies all assumptions of Proposition \ref{prop: gen_uqg} and Theorem \ref{thm: uqg_consequence_gen} and the conclusions there hold for \eqref{eq:qcp-primal}.
\end{cor}

\section{Conclusions}\label{sec: conclusion}
In this paper, we consider the conic programs \eqref{eq: p} and its convex composite optimization generalization \eqref{eq: P_gen}. As demonstrated in Figure \ref{fig:proof}, we establish a unified framework for showing local linear convergence for operator-splitting type methods \eqref{eq: iteration}, including PDHG and ADMM. Our framework consists of three components: (1) strict {\color{black}complementarity} and quadratic facial/subdifferential violation {\color{black}imply} uniform quadratic growth of augmented Lagrangian, (2) equivalence of three conditions for local linear convergence of \eqref{eq: iteration}, with one being the uniform quadratic growth of augmented Lagrangian,  (3) a unified argument for the local linear convergence of \eqref{eq: iteration} under one of the three equivalent conditions. 

\section*{Acknowledgments}
Haihao Lu is partially supported by AFOSR Grant No. FA9550-24-1-0051, ONR Grant No.
N000142412735 and a Sloan Research Fellowship. Jinwen
Yang is partially supported by AFOSR Grant No. FA9550-24-1-0051.





\bibliographystyle{amsplain}
\bibliography{ref-papers}

\appendix
\section{Proof of Proposition \ref{prop:local} and \ref{prop:local-restart}}\label{app:prove-local}
\begin{proof}[Proof of Proposition \ref{prop:local}]
We consider PDHG and ADMM, respectively.

\textbf{PDHG.}
For stepsize {\color{black}\(\eta\|A\|<1\)}, we have \(P\succ 0\), and hence $\lambda_{\min}(P)\|u\|^2
    \leq \|u\|_P^2
    \leq \lambda_{\max}(P)\|u\|^2$.
By the nonexpansiveness of PDHG iterates in the \(P\)-norm
\cite{chambolle2016ergodic}, for any \(\zsol\in\zsset\) and any
\(k\geq 0\), $\|z^k-\zsol\|_P
    \leq
    \|z^0-\zsol\|_P $.
Therefore,
\[
\begin{aligned}
    \|z^k-\zsol\|
    &\leq
    \frac{1}{\sqrt{\lambda_{\min}(P)}}\|z^k-\zsol\|_P \leq
    \frac{1}{\sqrt{\lambda_{\min}(P)}}\|z^0-\zsol\|_P \leq
    \sqrt{\frac{\lambda_{\max}(P)}{\lambda_{\min}(P)}}
    \|z^0-\zsol\| .
\end{aligned}
\]
Thus, letting $C_{\rm PDHG}:=
    \sqrt{\frac{\lambda_{\max}(P)}{\lambda_{\min}(P)}}$, we obtain
\[
    \|z^k-\zsol\|
    \leq
    C_{\rm PDHG}\|z^0-\zsol\|,
    \qquad \forall k\geq 0 .
\]

\textbf{ADMM.}
Note that \(A\) has full row rank, so that \(AA^{\top}\) is invertible.
By the Douglas--Rachford {\color{black}splitting (DRS)} interpretation of ADMM
\cite[Proposition~19]{eckstein2015understanding}, set
$h^k:=x^k-\rho A^\top y^k$ and
$h^*:=\xsol-\rho A^\top\ysol$.
The ADMM updates give $h^k=x^{k+1}+\rho v^{k+1}$ and
$x^{k+1}=\prox_{\rho f}(h^k)${\color{black}.
Set $\widehat x^{k+1}:=x^{k+1}-\rho(v^{k+1}+A^\top y^{k+1})$.
The} first-order condition for the $y$-update gives
{\color{black}$y^{k+1}\in\partial g(A\widehat x^{k+1})$.
By the convex subdifferential chain rule (valid since $A$ is
surjective), this implies
$\widehat x^{k+1}=\prox_{\rho(g\circ A)}
(2x^{k+1}-h^k)$}.
Consequently, {\color{black}$h^{k+1}=h^k+\widehat x^{k+1}-x^{k+1}=T(h^k)$, where}
\begin{equation}{\color{black}\label{eq: def_T}
    T(h):=h+\prox_{\rho(g\circ A)}(2\prox_{\rho f}(h)-h)-\prox_{\rho f}(h).}
\end{equation}
Indeed,
{\color{black}$T=\frac12\big[I+(2\prox_{\rho(g\circ A)}-I)(2\prox_{\rho f}-I)\big]$}
is firmly nonexpansive. The KKT conditions imply
$\xsol=\prox_{\rho f}(h^*)$ and $T(h^*)=h^*$.
Hence, for every $k\geq 1$,
\[
    \|h^k-h^*\|\leq\|h^0-h^*\|,\qquad
    \|x^k-\xsol\|\leq\|h^{k-1}-h^*\|\leq\|h^0-h^*\|.
\]
Using $h^k = x^k - \rho A^\top y^k$, we have $y^k
    =
    \frac{1}{\rho}(AA^{\top})^{-1}
    A(x^k - h^k)$. Thus, for any $k\geq 1$, we have
\[
{\color{black}\begin{aligned}
    \|y^k-\ysol\|
    &\leq
    \frac{\|(AA^{\top})^{-1}A\|}{\rho}
    \left(
        \|x^k-\xsol\|
        +
        \|h^k-h^*\|
    \right)  \leq     \frac{2\|(AA^{\top})^{-1}A\|}{\rho}
    \|h^0-h^*\| .
\end{aligned}}
\]
Therefore, recalling \(z^k=(x^k,y^k)\) and \(\zsol=(\xsol,\ysol)\), for all
\(k\geq 1\),
\begin{equation}\label{eq:zk-h1-compact}
    \|z^k-\zsol\|
    \leq
    \sqrt{
        1+
        \frac{4\|(AA^{\top})^{-1}A\|^2}{\rho^2}
    }
    \|h^0-h^*\| .
\end{equation}

It remains to bound \(\|h^0-h^*\|\). Using the definition of $h^0$, we have
\begin{equation}\label{eq:h1-z0-compact}
\begin{aligned}
    \|h^0-h^*\|
    &\leq
    \|x^0-\xsol\|+\rho \|A\| \|y^0-\ysol\|        \leq
    (1+\rho\|A\|)\|z^0-\zsol\| .
\end{aligned}
\end{equation}
Thus, combining the above two inequalities, with
\[
    C_{\rm ADMM}
    :=
    (1+\rho\|A\|)
    \sqrt{
        1+
        \frac{4\|(AA^{\top})^{-1}A\|^2}{\rho^2}
    },
\]
we have
\[
    \|z^k-\zsol\|
    \leq
    C_{\rm ADMM}\|z^0-\zsol\|,
    \qquad \forall k\geq 0 .
\]

{\color{black}The PDHG and ADMM estimates above establish Assumption~\ref{ass:local}.}
\end{proof}

\begin{proof}[Proof of Proposition \ref{prop:local-restart}]
    {\color{black}We give the proofs separately} for PDHG and ADMM.

    \textbf{PDHG.} Suppose $\|z^{n',k'}-{\zsol}\|_P\leq \|z^{n',0}-{\zsol}\|_P\leq \|z^{0,0}-{\zsol}\|_P$ for any $n'\leq n-1$ and $k'\leq k^*$, and $n'=n$ with $k'\leq k$. Then consider $z^{n,k+1}$ and we have
    \begin{equation*}
    \begin{aligned}
        \|z^{n,k+1}-{\zsol}\|_P&\leq {\frac{k+1}{k+2}\|z^{n,k}-{\zsol}\|_P+\frac{1}{k+2}\|z^{n,0}-{\zsol}\|_P}\leq\|z^{n,0}-{\zsol}\|_P\leq \|z^{0,0}-{\zsol}\|_P \ . 
    \end{aligned}
\end{equation*}
Thus by induction we have for any $n$ and any $0\leq k\leq k^*$, we have $\|z^{n,k}-{\zsol}\|_P\leq \|z^{0,0}-{\zsol}\|_P$, and therefore, let $C_{\rm PDHG}:=\sqrt{\frac{\lambdamax(P)}{\lambdamin(P)}}$, it holds that
\begin{equation*}
    \|z^{n,k}-{\zsol}\|\leq \sqrt{\frac{1}{\lambdamin(P)}}\|z^{n,k}-{\zsol}\|_P\leq \sqrt{\frac{1}{\lambdamin(P)}}\|z^{0,0}-{\zsol}\|_P\leq \sqrt{\frac{\lambdamax(P)}{\lambdamin(P)}}\|z^{0,0}-{\zsol}\|=C_{\rm PDHG}\|z^{0,0}-{\zsol}\| \ .
\end{equation*}
Since $z^{n+1,0}=\tilde z^{n,k^*+1}$, nonexpansiveness gives the same bound for the next outer anchor.

        \textbf{ADMM.} Define
    $h^{n,k}:=x^{n,k}-\rho A^\top y^{n,k}$ and
    $\tilde h^{n,k+1}:=\tilde x^{n,k+1}-\rho A^\top\tilde y^{n,k+1}$.
    By the ADMM update, $\tilde h^{n,k+1}=T(h^{n,k})$, where the nonexpansive DRS opeartor $T$ is defined in \eqref{eq: def_T}, and by
    the averaging step,
    $h^{n,k+1}=\frac{k+1}{k+2}\tilde h^{n,k+1}
       +\frac{1}{k+2}h^{n,0}$.
    For $h^*:=\xsol-\rho A^\top\ysol\in\operatorname{Fix}(T)$,
    suppose $\|h^{n',k'}-h^*\|\leq \|h^{n',0}-h^*\|
    \leq\|h^{0,0}-h^*\|$ for all previous inner and outer steps.
    Then
    \begin{equation*}
        \begin{aligned}
            \|h^{n,k+1}-h^*\|&= \left\|\frac{k+1}{k+2}\tilde h^{n,k+1}+\frac{1}{k+2}h^{n,0}-h^*\right\|\leq \frac{k+1}{k+2}\|\tilde h^{n,k+1}-h^*\|+\frac{1}{k+2}\|h^{n,0}-h^*\|\\
            &\leq \frac{k+1}{k+2}\|h^{n,k}-h^*\|+\frac{1}{k+2}\|h^{n,0}-h^*\|\leq \|h^{n,0}-h^*\| \ ,
        \end{aligned}
    \end{equation*}
    where the second inequality uses the {\color{black}nonexpansiveness of the} DRS iteration, and the last one follows from {\color{black}the induction hypothesis}.
    Thus by induction, for all $h^{n,k}$ we have $\|h^{n,k}-h^*\|\leq \|h^{n,0}-h^*\|\leq \cdots \leq \|h^{0,0}-h^*\|$, together with
        \begin{equation*}
        \|\tilde h^{n,k+1}-h^*\|
        \leq \|h^{n,k}-h^*\|\leq \|h^{0,0}-h^*\|.
    \end{equation*}
    Moreover, $\tilde x^{n,k+1}=\prox_{\rho f}(h^{n,k})$ and
    $\tilde y^{n,k+1}=\rho^{-1}(AA^\top)^{-1}A
       (\tilde x^{n,k+1}-\tilde h^{n,k+1})$.
    Thus the estimates in Proposition~\ref{prop:local} bound each
    shadow iterate by
    $C_{\rm ADMM}\|z^{0,0}-\zsol\|$; the same bound holds for
    the averaged and restarted iterates by convexity.

    {\color{black}These estimates prove Assumption~\ref{ass:local-restart}}.
\end{proof}

\end{document}